\pdfoutput=1
\newenvironment{example}{}{}

\documentclass[graybox,nospthms]{svmult}

\usepackage[utf8]{inputenc}
\usepackage[T1]{fontenc}

\usepackage{type1cm}        
\usepackage{makeidx}         
\usepackage{graphicx}        
\usepackage{multicol}        
\usepackage[bottom]{footmisc}
\usepackage{amssymb,amsmath}
\usepackage{stmaryrd}
\usepackage{relsize}
\usepackage[colorlinks,citecolor=blue]{hyperref}
\usepackage[capitalize,nameinlink]{cleveref}

\usepackage{newtxtext}       %
\usepackage{newtxmath}       

\makeindex             
                       
\newcommand\norm[2]{\left\Vert#1\right\Vert_{#2}}
\newcommand\abs[2]{\left\vert#1\right\vert_{#2}}
\newcommand\dual[3]{\left\langle #1, #2\right\rangle_{#3}}
\newcommand\linop[2]{\mathbb L\left[#1,#2\right]}

\newcommand\N{\mathbb{N}}
\newcommand\R{\mathbb{R}}

\newcommand\XX{\mathcal{X}}
\newcommand\YY{\mathcal{Y}}
\newcommand\ZZ{\mathcal{Z}}
\newcommand\WW{\mathcal{W}}

\newcommand\opA{\mathtt{A}}
\newcommand\opB{\mathtt{B}}

\newcommand\opS{\mathtt{S}}

\newcommand\mto{\rightrightarrows}

\newcommand{\cl}{\operatorname{cl}}
\newcommand{\bd}{\operatorname{bd}}

\newcommand{\conv}{\operatorname{conv}}
\newcommand{\cone}{\operatorname{cone}}
\newcommand{\dom}{\operatorname{dom}}
\newcommand{\gph}{\operatorname{gph}}

\newcommand{\weakly}{\rightharpoonup}

\theoremstyle{plain}
\theoremheaderfont{\upshape\bfseries}
\theorembodyfont{\upshape}
\theoremsymbol{\ensuremath{\triangle}}

\newtheorem{theorem}{Theorem}[section]
\newtheorem{lemma}[theorem]{Lemma}

\newtheorem{definition}[theorem]{Definition}
\newtheorem{assumption}[theorem]{Assumption}
\newtheorem{remark}[theorem]{Remark}

\allowdisplaybreaks

\apptocmd\example{\ignorespaces}{}{}

\newcommand\doi[1]{\href{https://dx.doi.org/#1}{doi: \nolinkurl{#1}}}

\crefname{assumption}{assumption}{assumptions}
\Crefname{assumption}{Assumption}{Assumptions}

\begin{document}

\title*{Bilevel optimal control: existence results and stationarity conditions}
\author{Patrick Mehlitz and Gerd Wachsmuth}
\institute{
	Patrick Mehlitz \at 
	Brandenburgische Technische Universit\"at Cottbus Senftenberg, Chair of Optimal Control, 
	\email{mehlitz@b-tu.de}
		\and 
	Gerd Wachsmuth \at 
	Brandenburgische Technische Universit\"at Cottbus Senftenberg, Chair of Optimal Control,
	\email{wachsmuth@b-tu.de}
}
%
%
\maketitle

\abstract*{
	The mathematical modeling of numerous real-world applications 
	results in hierarchical optimization problems with two decision makers
	where at least one of them has to solve an optimal control problem of
	ordinary or partial differential equations.
	Such models are referred to as bilevel optimal control problems.
	Here, we first review some different features of bilevel optimal control
	including important applications, existence results, solution approaches,
	and optimality conditions.
	Afterwards, we focus on a specific problem class
	where parameters appearing in the objective functional of an optimal
	control problem of partial differential equations have to
	be reconstructed. After verifying the existence of solutions, necessary
	optimality conditions are derived by exploiting the optimal value function
	of the underlying parametric optimal control problem in the context of
	a relaxation approach. 
}

\abstract{
	The mathematical modeling of numerous real-world applications 
	results in hierarchical optimization problems with two decision makers
	where at least one of them has to solve an optimal control problem of
	ordinary or partial differential equations.
	Such models are referred to as bilevel optimal control problems.
	Here, we first review some different features of bilevel optimal control
	including important applications, existence results, solution approaches,
	and optimality conditions.
	Afterwards, we focus on a specific problem class
	where parameters appearing in the objective functional of an optimal
	control problem of partial differential equations have to
	be reconstructed. After verifying the existence of solutions, necessary
	optimality conditions are derived by exploiting the optimal value function
	of the underlying parametric optimal control problem in the context of
	a relaxation approach. 
}

\section{What is bilevel optimal control?}\label{sec:introduction}

A bilevel programming problem is a hierarchical optimization problem of two decision
makers where the objective functional as well as the feasible set of the so-called
upper level decision maker (or leader) depend implicitly on the solution set of a 
second parametric optimization problem which will be called lower level (or follower's)
problem. Both decision makers act as follows: First, the leader chooses an instance
from his feasible set which then serves as the parameter in the follower's problem.
Thus, the follower is in position to solve his problem and passes an optimal 
solution back to the leader who now may compute the associated value of the 
objective functional. As soon as the lower level solution set is not a singleton
for at least one value of the upper level variable, problems of this type may be ill-posed
which is why different solution concepts including the so-called optimistic and
pessimistic approach have been developed. Bilevel programming problems generally suffer
from inherent lacks of convexity, regularity, and smoothness which makes them
theoretically challenging.
The overall concept of bilevel 
optimization dates back to \cite{vonStackelberg1934} where this problem class
is introduced in the context of economical game theory. More than 80 years later, bilevel
programming is one of the hottest topics in mathematical optimization since numerous
real-world applications can be transferred into models of bilevel structure. 
A detailed introduction to bilevel programming can be found in the monographs
\cite{Bard1998,Dempe2002,DempeKalashnikovPerezValdesKalashnykova2015,ShimizuIshizukaBard1997}
while a satisfying overview of existing literature is given in \cite{Dempe2018}
where more than 1350 published books, PhD-theses, and research articles are listed.

Optimal control of ordinary or partial differential equations (ODEs and PDEs, respectively)
describes the task of identifying input quantities which control the state function
of the underlying differential equation such that a given cost functional is minimized,
see \cite{HinzePinnauUlbrichUlbrich2009,LewisVrabieSyrmos2012,Troeltzsch2009,Troutman1996}
for an introduction to this topic. Noting that the decision variables are elements of
suitable function spaces, optimal control is a particular field of programming in
(infinite-dimensional) Banach spaces, see \cite{BonnansShapiro2000}.

In bilevel optimal control, bilevel programming problems are considered where at least one
decision maker has to solve an optimal control problem. Thus, we are facing the intrinsic
difficulties of bilevel optimization \emph{and} optimal control when investigating this problem class.
Naturally, one may subdivide bilevel optimal control problems into three subclasses depending on
which decision maker has to perform optimal control. Each of these problem classes 
appears in practice and has to be tackled with different techniques in order to infer optimality
conditions or solution algorithms.

The situation where only the upper level decision maker has to solve an optimal control problem
of ordinary differential equations while the lower level problem explicitly depends on the terminal state
of the leader's state variable has been considered in \cite{BenitaDempeMehlitz2016,BenitaMehlitz2016}.
Problems of this type arise from the topic of gas balancing in energy networks, 
see \cite{KalashnikovBenitaMehlitz2015}, and can be investigated by combining tools from finite-dimensional
parametric optimization and standard optimal control.
The situation where parameters within an optimal control problem have to be estimated or reconstructed by
certain measurements is a typical example of a bilevel optimal control problem where only the lower
level decision maker has to solve an optimal control problem. This particular instance of bilevel optimal
control may therefore be also called inverse optimal control. 
In \cite{AlbrechtLeiboldUlbrich2012,AnbrechtPassenbergSobotkaPeerBussUlbrich2010,AlbrechtUlbrich2017,Hatz2014,
MombaurTruongLaumond2010},
inverse optimal control problems of ODEs are considered in the context of human locomotion.
Some more theoretical results for such problems are presented in \cite{HatzSchloederBock2012,Ye1995,Ye1997}.
First steps regarding the inverse optimal
control of PDEs have been done recently in the papers 
\cite{DempeHarderMehlitzWachsmuth2019,HarderWachsmuth2019,HollerKunischBarnard2018}.
The paper \cite{PalagachevGerdts2017} deals with the scheduling of multiple agents which are controlled
at the lower level stage. In \cite{FischLenzHolzapfelSachs2012}, the authors discuss a bilevel
optimal control problem where airplanes are controlled at multiple lower levels in order to increase the
fairness in air racing.
Finally, it is possible that leader and follower have to solve an optimal control problem.
This setting has been discussed theoretically in 
\cite{Carlson2013,Mehlitz2016,MehlitzWachsmuth2016,PalagachevGerdts2016}. 
Underlying applications arise e.g.\ when time-dependent coupling of container crane movements
is under consideration, see \cite{Knauer2012,KnauerBueskens2010}.

The optimal control of (quasi-) variational inequalities ((Q)VIs) seems to be closely related to the subject of
bilevel optimal control since the underlying variational problem, which assigns to each control the 
uniquely determined state function, can be modeled as a parametric optimization problem in function spaces. 
Those problems are of hierarchical structure, but neither leader nor
follower has to solve an optimal control problem in the classical meaning.
In the seminal work \cite{Mignot1976},
Mignot shows that the control-to-state map of an elliptic VI in the Sobolev space
$H_0^1(\Omega)$
is directionally differentiable, and
(in the absence of control constraints)
this leads to an optimality system of strong-stationarity-type.
If control constraints are present,
one typically uses a regularization approach for the derivation
of optimality conditions.
This idea dates back to \cite{Barbu1984}
and we refer to \cite{SchielaWachsmuth2013} for a modern treatment.
Finally, we would like to mention 
that a comparison of several optimality systems
and further references regarding this topic can be found in 
\cite{HarderWachsmuth2018a}.

\section{Notation and preliminaries}\label{sec:preliminaries}

Let us briefly recall some essentials of functional analysis we are going to exploit.
For a (real) Banach space $\XX$, $\norm{\cdot}{\XX}\colon\XX\to\R$ denotes its norm.
Furthermore, $\XX^\star$ represents the topological dual of $\XX$. 
We use $\dual{\cdot}{\cdot}{\XX}\colon\XX^\star\times\XX\to\R$ in order to denote
the associated dual pairing. For a sequence $\{x_k\}_{k\in\N}\subset\XX$ and some
point $\bar x\in \XX$, strong and weak convergence of $\{x_k\}_{k\in\N}$ to $\bar x$
will be represented by $x_k\to\bar x$ and $x_k\weakly\bar x$, respectively.
Recall that in a finite-dimensional Banach space $\XX$, the concepts of strong and
weak convergence coincide.
A functional $J\colon\XX\to\R$ is said to be weakly sequentially lower (upper)
semicontinuous at $\bar x$, whenever
\[
	x_k\weakly\bar x\,\Longrightarrow\,j(\bar x)\leq\liminf\limits_{k\to\infty}j(x_k)
	\qquad
	\left(x_k\weakly\bar x\,\Longrightarrow\,j(\bar x)\geq\limsup\limits_{k\to\infty}j(x_k)\right)
\]
holds
for all sequences $\{x_k\}_{k\in\N}\subset X$. We say that $j$ is weakly sequentially lower
(upper) semicontinuous if it possesses this property at each point from $\XX$.
It is well known that convex and continuous functionals are weakly sequentially
lower semicontinuous.
If the canonical embedding $\XX\ni x\mapsto\dual{\cdot}{x}{\XX}\in\XX^{\star\star}$
is an isomorphism, then $\XX$ is said to be reflexive.
The particular Banach space $\R^n$ is equipped with the Euclidean norm $\abs{\cdot}{2}$.
Furthermore, we use $x\cdot y$ to represent the Euclidean inner product in $\R^n$.

A set $A\subset\XX$ is said to be weakly sequentially closed whenever the weak limits of
all weakly convergent sequences from $A$ belong to $A$ as well. We note that closed
and convex sets are weakly sequentially closed. We call $A$ weakly sequentially compact
whenever each sequence from $A$ possesses a weakly convergent subsequence whose limit
belongs to $A$. Each bounded, closed, and convex subset of a reflexive Banach
space is weakly sequentially compact.

For a second Banach space $\YY$, $\linop{\XX}{\YY}$ is used to denote the Banach
space of all bounded linear operators mapping from $\XX$ to $\YY$. 
For $\opA\in\linop{\XX}{\YY}$, $\opA^\star\in\linop{\YY^\star}{\XX^\star}$ denotes
its adjoint.
If $\XX\subset\YY$ holds while the associated identity in $\linop{\XX}{\YY}$
is continuous, then $\XX$ is said to be continuously embedded into $\YY$ which
will be denoted by $\XX\hookrightarrow\YY$.
Whenever the identity is compact, the embedding $\XX\hookrightarrow\YY$ is called compact.
For a set-valued mapping $\Gamma\colon\XX\mto\YY$, 
$\gph\Gamma:=\{(x,y)\in \XX\times\YY\,|\,y\in\Gamma(x)\}$ and
$\dom\Gamma:=\{x\in\XX\,|\,\Gamma(x)\neq\varnothing\}$
represent the graph and the domain of $\Gamma$, respectively.

Let $A\subset \XX$ be nonempty and convex. Then, the closed, convex cone
\[
	A^\circ:=\left\{
		x^\star\in\XX^\star\,\middle|\,\forall x\in A\colon\;\dual{x^\star}{x}{\mathcal X}\leq 0
			\right\}
\]
is called the polar cone of $A$. For a fixed point $\bar x\in A$, 
$\mathcal N_A(\bar x):=(A-\{\bar x\})^\circ$
is referred to as the normal cone (in the sense of convex analysis) to $A$ at $\bar x$.
For the purpose of completeness, let us set $\mathcal N_A(\hat x):=\varnothing$
for all $\hat x\in\mathcal X\setminus A$.
Note that whenever $C\subset\XX$ is a closed, convex cone satisfying $\bar x\in C$,
then we have the relation 
$\mathcal N_C(\bar x)=C^\circ\cap\{x^\star\in\mathcal X^\star\,|\,\langle x^\star,x\rangle_{\XX}=0\}$.

Detailed information on the function spaces we are going to exploit can be found in
the monograph \cite{AdamsFournier2003}.


\section{Bilevel programming in Banach spaces}

Let us consider the bilevel programming problem
\begin{equation}\label{eq:upper_level_abstract}\tag{BPP}
	\begin{split}
		F(x,z)&\,\to\,\min\limits_{x,z}\\
		x&\,\in\,X_\textup{ad}\\
		z&\,\in\,\Psi(x),
	\end{split}
\end{equation}
where $\Psi\colon\mathcal X\mto\mathcal Z$ is the solution mapping of the parametric optimization problem
\begin{equation}\label{eq:lower_level_abstract}
	\begin{split}
		f(x,z)&\,\to\,\min\limits_z\\
		z&\,\in\,\Gamma(x).
	\end{split}
\end{equation}
Note that we minimize the objective functional in \eqref{eq:upper_level_abstract} w.r.t.\ both
variables which is related to the so-called optimistic approach of bilevel programming.
In this section, we first want to discuss the existence of optimal solutions associated with
\eqref{eq:upper_level_abstract}. Afterwards, we briefly discuss possible approaches which can
be used to infer optimality conditions for this problem class.

\subsection{Existence theory}

In this section, we aim to characterize situations where \eqref{eq:upper_level_abstract}
possesses optimal solutions. 
Noting that compact sets are generally rare in infinite-dimensional spaces, one cannot rely
on classical existence results from bilevel programming. Indeed, compactness assumptions
on the feasible sets have to be relaxed in order to guarantee applicability of possible results.
As a consequence, we need to demand more restrictive properties than (lower semi-) continuity of
the appearing objective functionals in order to balance things in a reasonable way. 
One may check e.g.\ \cite{Jahn1996} for a detailed discussion of existence theory for 
optimization problems in Banach spaces. 
Particularly, it is presented that each weakly sequentially lower semicontinuous functional
achieves its minimum over a nonempty and weakly sequentially compact set.
The above remarks justify the subsequently stated general assumptions of this section.
\begin{assumption}\label{ass:abstract_model}
	We consider Banach spaces $\XX$ and $\ZZ$.
	The objective functionals $F,f\colon\XX\times\ZZ\to\R$ are weakly sequentially lower semicontinuous.
	The set $X_\textup{ad}\subset\XX$ is assumed to be nonempty and weakly sequentially compact.
	Furthermore, $\Gamma\colon\XX\mto\ZZ$ is a set-valued mapping with $X_\textup{ad}\subset\dom\Gamma$
	such that $(X_\textup{ad}\times\ZZ)\cap\gph\Gamma$ is weakly sequentially compact.
\end{assumption}

In the setting where $\XX$ and $\ZZ$ are finite-dimensional, e.g.\ instances of $\R^n$, the
above assumptions reduce to the lower semicontinuity of the objective functionals as well as
some compactness assumptions on $X_\textup{ad}$ and $\gph\Gamma$ which is rather
standard in bilevel programming, see e.g.\ \cite{Dempe2002}. 
For our upcoming analysis, we will exploit the function $\varphi\colon X_\textup{ad}\to\overline\R$
defined by
\begin{equation}\label{eq:value_function}
	\forall x\in X_\textup{ad}\colon\quad
		\varphi(x):=\inf\limits_z\{f(x,z)\,|\,z\in\Gamma(x)\}.
\end{equation}
By definition, $\varphi$ assigns to each parameter $x\in X_\textup{ad}$ the optimal function value
of the lower level problem \eqref{eq:lower_level_abstract}. 
\Cref{ass:abstract_model} guarantees that the infimal  
value $\varphi(x)$ is actually attained (i.e.\ $\Psi(x)\neq\varnothing$) since for all $x\in X_\text{ad}$, 
$f(x,\cdot)\colon\ZZ\to\R$ is weakly sequentially lower semicontinuous while $\Gamma(x)$
is nonempty and weakly sequentially compact.

Below, we need to study the (upper) semicontinuity properties of $\varphi$.
In order to do that, we need to address some continuity properties of the mapping $\Gamma$,
see \cite{HerzogSchmidt2012}.
\begin{definition}\label{def:isc}
	Fix $(\bar x,\bar z)\in\gph\Gamma$. Then, $\Gamma$ is called
	inner semicontinuous (weakly-weakly inner semicontinuous) at $(\bar x,\bar z)$ 
	if for each sequence $\{x_k\}_{k\in\N}\subset\dom\Gamma$ satisfying $x_k\to\bar x$ ($x_k\weakly\bar x$),
	there exists a sequence $\{z_k\}_{k\in\N}\subset\ZZ$ satisfying
	$z_k\in\Gamma(x_k)$ for all $k\in\N$ as well as $z_k\to\bar z$ ($z_k\weakly\bar z$).
\end{definition}

It needs to be noted that the concepts of inner and lower semicontinuity of set-valued mappings,
see \cite{BankGuddatKlatteKummerTammer1983}, are closely related. 
Particularly, the lower semicontinuity of $\Gamma$ at some
point $\bar x\in\dom\Gamma$ is equivalent to its inner semicontinuity at all points $(\bar x,z)$ with $z\in\Gamma(\bar x)$.

In the particular situation where the mapping $\Gamma$ is characterized via smooth generalized inequalities,
there is an easy criterion which is sufficient for inner semicontinuity.
\begin{remark}\label{rem:lower_level_abstract_inequality_constraints}
	We assume that there exists a continuously Fr\'{e}chet differentiable function $g\colon\XX\times\ZZ\to\WW$,
	where $\WW$ is a Banach space, and some nonempty, closed, convex set $C\subset\WW$ such that $\Gamma$ is given by
	\[
		\forall x\in\XX\colon\quad
		\Gamma(x):=\{z\in\ZZ\,|\,g(x,z)\in C\}.
	\]
	For fixed $\bar z\in\Gamma(\bar x)$, we assume that the condition
	\begin{equation}\label{eq:RCQ}
		g'_z(\bar x,\bar z)\ZZ-\cone(C-\{g(\bar x,\bar z)\})=\WW
	\end{equation}
	is valid. Then, $\Gamma$ is inner semicontinuous at $(\bar x,\bar z)$, 
	see e.g.\ \cite[Section~2.3.3]{BonnansShapiro2000}.
	We note that \eqref{eq:RCQ} often is referred to as Robinson's constraint qualification, see \cite{Robinson1976},
	or Kurcyusz--Zowe constraint qualification, see \cite{ZoweKurcyusz1979}. In the setting of finite-dimensional
	nonlinear parametric optimization, this condition simply reduces to the Mangasarian--Fromovitz constraint qualification,
	see \cite{BonnansShapiro2000} for details.
	Let us note that \eqref{eq:RCQ} trivially holds whenever the operator $g'_z(\bar x,\bar z)$ is surjective.
\end{remark}

We note that weakly-weakly inner semicontinuity of $\Gamma$ is inherent whenever this map is actually constant.
A nontrivial situation is described in the following example.
\begin{example}{Terminal state dependence of lower level}\label{ex:weakly_weakly_inner_semicontinuity_nontrivial}
	For some time interval $I:=(0,T)$ and some natural number $n$, we consider the Hilbert space $\XX:=H^1(I;\R^n)$. 
	Clearly, the embedding $\XX\hookrightarrow C(\overline I;\R^n)$ is compact, see \cite{AdamsFournier2003}.
	This means that the evaluation operator $\XX\ni x\mapsto x(T)\in\R^n$ is well-defined and compact as well.
	
	For some set-valued mapping $\Upsilon\colon\R^n\mto\ZZ$, we define $\Gamma(x):=\Upsilon(x(T))$ for
	all $x\in\XX$. The above observation implies that $\Gamma$ is weakly-weakly inner semicontinuous 
	at $(\bar x,\bar z)\in\gph\Gamma$ whenever $\Upsilon$ is inner semicontinuous at $(\bar x(T),\bar z)$
	and the latter can be guaranteed via standard assumptions, 
	see e.g.\ \Cref{rem:lower_level_abstract_inequality_constraints}.
	
	The setting in this example reflects the situation of time-dependent coupling
	between upper and lower level, see \cite[Section~5]{Knauer2012},
	or a finite-dimensional lower level problem depending only on the terminal value
	of the leader's state variable, see \cite{BenitaDempeMehlitz2016,BenitaMehlitz2016,KalashnikovBenitaMehlitz2015}.
	
	It needs to be noted that the analysis of the above situation can be extended to cases
	where $\XX$ is a function space over some domain $\Omega\subset\R^d$ which is embedded compactly
	into $C(\overline\Omega)$, and the function of interest is evaluated at finitely many points
	from $\overline\Omega$. This applies to the setting $\XX:=H^2(\Omega)$ where $\Omega$ is a bounded
	Lipschitz domain, see \cite{AdamsFournier2003}.
\end{example}

In the following lemma, which is inspired by \cite[Theorem~2.5]{HerzogSchmidt2012}, we study upper
semicontinuity properties of the function $\varphi$. This will be useful in order to infer closedness
properties of the feasible set associated with \eqref{eq:upper_level_abstract}.
\begin{lemma}\label{lem:upper_semicontinuity_phi}
	Fix some point $\bar x\in X_\textup{ad}$ and $\bar z\in\Psi(\bar x)$.
	\begin{enumerate}
		\item Assume that $f$ is weakly sequentially upper semicontinuous 
			at $(\bar x,\bar z)$ while $\Gamma$ is weakly-weakly inner semicontinuous at $(\bar x,\bar z)$.
			Then, $\varphi$ is weakly sequentially upper semicontinuous at $\bar x$.
		\item Let $\XX$ be finite-dimensional. Assume that $f$ is upper semicontinuous at $(\bar x,\bar z)$
			while $\Gamma$ is inner semicontinuous at $(\bar x,\bar z)$.
			Then, $\varphi$ is upper semicontinuous at $\bar x$.
	\end{enumerate}
\end{lemma}
	\begin{proof}
		We only verify the first statement of the lemma. The second one can be shown using analogous arguments.
		
		Let $\{x_k\}_{k\in\N}\subset X_\textup{ad}$ be a sequence satisfying $x_k\weakly\bar x$.
		Exploiting the weakly-weakly inner semicontinuity of $\Gamma$ at $(\bar x,\bar z)$, we find a sequence
		$\{z_k\}_{k\in\N}\subset\ZZ$ satisfying $z_k\in\Gamma(x_k)$ for all $k\in\N$ and $z_k\weakly \bar z$.
		By definition, $\varphi(x_k)\leq f(x_k,z_k)$ holds for all $k\in\N$. Now, the weak sequential upper
		semicontinuity of $f$ at $(\bar x,\bar z)$ yields
		\[
			\limsup\limits_{k\to\infty}\varphi(x_k)
			\leq
			\limsup\limits_{k\to\infty}f(x_k,z_k)
			\leq
			f(\bar x,\bar z)
			=
			\varphi(\bar x),
		\]
		and this shows the claim.
	\end{proof}
	
Now, we exploit the above lemma in order to infer the existence of optimal solutions to \eqref{eq:upper_level_abstract}.
\begin{theorem}\label{thm:existence_abstract}
	In each of the settings described below, \eqref{eq:upper_level_abstract} possesses an optimal solution.
	\begin{enumerate}
		\item The mapping $f$ is weakly sequentially upper semicontinuous on $X_\textup{ad}\times\ZZ$ while $\Gamma$ is 
			weakly-weakly inner semicontinuous on $X_\textup{ad}\times\ZZ$.
		\item The Banach space $\XX$ is finite-dimensional.
			The mapping $f$ is upper semicontinuous on $X_\textup{ad}\times\ZZ$ while $\Gamma$ is inner semicontinuous
			on $X_\textup{ad}\times\ZZ$.
	\end{enumerate}
\end{theorem}
\begin{proof}
	Again, we only show the theorem's first assertion.
	
	For the proof, we just need to verify that the feasible set $(X_\textup{ad}\times\ZZ)\cap\gph\Psi$ 
	of \eqref{eq:upper_level_abstract} is nonempty and	weakly sequentially compact since the
	objective $F$ is supposed to be weakly sequentially lower semicontinuous.
	Noting that $\Psi(x)\neq\varnothing$ holds true for all $x\in X_\textup{ad}$, the nonemptiness of
	$(X_\textup{ad}\times\ZZ)\cap\gph\Psi$ is obvious.
	
	Let $\{(x_k,z_k)\}_{k\in\N}\subset(X_\textup{ad}\times\ZZ)\cap\gph\Psi$ be an arbitrary sequence.
	Clearly, we have $\{(x_k,z_k)\}_{k\in\N}\subset(X_\textup{ad}\times\ZZ)\cap\gph\Gamma$
	and by \Cref{ass:abstract_model}, there exists a subsequence 
	(without relabeling) converging weakly to $(\bar x,\bar z)\in(X_\textup{ad}\times\ZZ)\cap\gph\Gamma$.
	Now, the definition of the function $\varphi$ and \Cref{lem:upper_semicontinuity_phi} yield
			\[
				\varphi(\bar x)
				\leq 
				f(\bar x,\bar z)
				\leq
				\liminf\limits_{k\to\infty}f(x_k,z_k)
				=
				\liminf\limits_{k\to\infty}\varphi(x_k)
				\leq
				\limsup\limits_{k\to\infty}\varphi(x_k)
				\leq
				\varphi(\bar x),
			\]
	which shows $\varphi(\bar x)=f(\bar x,\bar z)$, i.e.\ $\bar z\in\Psi(\bar x)$ follows.
	This yields that the point $(\bar x,\bar z)$ belongs to $(X_\textup{ad}\times\ZZ)\cap\gph\Psi$,
	and this shows the claim.
\end{proof}

Let us apply the above theory to some example problems from bilevel optimal control.
\begin{example}{Inverse nonregularized control of Poisson's equation}\label{ex:optimal_control_of_bang_bang_system}
	Let $\Omega\subset\R^d$ be a bounded domain with Lipschitz boundary $\bd\Omega$.
	For fixed parameters $x^w\in\R^n$ and $x^s\in L^2(\Omega)$, we consider the optimal
	control of Poisson's equation
	\[
		\begin{aligned}
			\tfrac12\norm{y-\mathsmaller\sum\nolimits_{i=1}^nx^w_if^i}{L^2(\Omega)}^2&\,\to\,\min\limits_{y,u}&&&\\
			-\Delta y&\,=\,x^s+u&&&\\
			u_a\,\leq\,u&\,\leq\,u_b&&\text{a.e.\ on }\Omega			
		\end{aligned}
	\]
	where $f^1,\ldots,f^n\in L^2(\Omega)$ are fixed form functions and
	$u_a,u_b\in L^2(\Omega)$ are given functions satisfying $u_a<u_b$ almost everywhere on $\Omega$.
	The variables $y$ and $u$ are chosen from the respective spaces $H^1_0(\Omega)$ and $L^2(\Omega)$.
	The underlying PDE has to be understand in weak sense in $H^{-1}(\Omega):=H^1_0(\Omega)^\star$.
	In this regard, the source term $x^s+u$ from $L^2(\Omega)$ is embedded into $H^{-1}(\Omega)$, implicitly.
	Noting that no regularization term w.r.t.\ the control appears in the objective functional,
	optimal controls are promoted which take values only at the lower and upper bound $u_a$ and $u_b$,
	and such controls are referred to as bang-bang, see \cite{Troeltzsch2009}.
	
	Let $\Psi\colon \R^n\times L^2(\Omega)\mto H^1_0(\Omega)\times L^2(\Omega)$ be the solution map associated
	with the above optimal control problem. In the superordinate upper level problem, we aim to identify
	the lower level desired state via correct choice of the weights $x^w\in\R^n$ 
	and constant source $x^s\in L^2(\Omega)$ from a nonempty, closed, convex,
	and bounded set $X_\textup{ad}\subset \R^n\times L^2(\Omega)$ such that a resulting optimal solution is close
	to observed data functions $y_\textup o,u_\textup o\in L^2(\Omega)$. 
	A suitable model for this program is given by
	\[
		\begin{split}
			\tfrac12\norm{y-y_\textup o}{L^2(\Omega)}^2+\tfrac12\norm{u-u_\textup o}{L^2(\Omega)}^2&\,\to\,\min\limits_{x,y,u}\\
			(x^w,x^s)&\,\in\,X_\textup{ad}\\
			(y,u)&\,\in\,\Psi(x^w,x^s).
		\end{split}
	\]
	Due to continuity and convexity of the appearing objective functionals, they are weakly sequentially lower
	semicontinuous. Furthermore, the compactness of $H^1_0(\Omega)\hookrightarrow L^2(\Omega)$ even guarantees
	that the objective of the lower level problem is weakly sequentially continuous.
	The set $X_\textup{ad}$ is nonempty and weakly sequentially compact by assumption.
	Exploiting the linearity and continuity of the solution operator $(-\Delta)^{-1}$
	of Poisson's equation, it is not difficult to see that the graph of the lower level feasible set mapping
	$\Gamma$ is convex and closed. The boundedness of $X_\textup{ad}$ ensures the boundedness
	of $(X_\textup{ad}\times H^1_0(\Omega)\times L^2(\Omega))\cap\gph\Gamma$, and it is not difficult to see that
	this set is weakly sequentially compact as well.
	Using the properties of $(-\Delta)^{-1}$, it is easy to see that $\Gamma$ is weakly-weakly inner semicontinuous
	at all points of its graph. Now, \Cref{thm:existence_abstract} yields the existence of a solution to
	the bilevel optimal control problem under consideration.
\end{example}
\begin{example}{Optimal control of ODEs with terminal penalty cost}
	For a fixed given vector $\xi\in\R^n$ of parameters, we consider the parametric optimization problem
	\begin{equation}\label{eq:finite_dimensional_lower_level}
		\begin{aligned}
			j(\xi,z)&\,\to\,\min\limits_z\\
			g(\xi,z)&\,\leq\,0
		\end{aligned}
	\end{equation}
	where $j\colon\R^n\times\R^m\to\R$ is continuous and $g\colon\R^n\times\R^m\to\R^k$ is continuously 
	differentiable. Furthermore, we assume that $\Upsilon(\xi):=\{y\in\R^m\,|\,g(\xi,y)\leq 0\}$ is
	nonempty for each $\xi\in\R^n$, that $\bigcup_{\xi\in\R^n}\Upsilon(\xi)$ is
	bounded, and that the Mangasarian--Fromovitz constraint qualification holds at all feasible points
	associated with \eqref{eq:finite_dimensional_lower_level}.
	
	The associated upper level problem shall be given by
	\begin{equation}\label{eq:ODE_control_terminal_penalty_cost}
		\begin{aligned}
			\tfrac12\norm{x-x_\textup{d}}{L^2(I;\R^n)}^2+\tfrac\sigma2\norm{u}{L^2(I;\R^p)}^2+J(x(T),y)
				&\,\to\,\min\limits_{x,u,y}\\
				\dot x -Ax-Bu&\,=\,0\\
				x(0)&\,=\,0\\
				u_a\,\leq\,u&\,\leq\,u_b\\
				y&\,\in\,\Psi(x)
		\end{aligned}
	\end{equation}
	where $I:=(0,T)$ is a time interval, $x_\textup{d}\in L^2(I;\R^n)$ is a desired state,
	$\sigma\geq 0$ is a regularization parameter, $J\colon\R^n\times\R^m\to\R$ is lower
	semicontinuous, $A\in\R^{n\times n}$ as well as $B\in\R^{n\times p}$ are fixed
	matrices, $u_a,u_b\in L^2(I;\R^p)$ are fixed functions satisfying $u_a<u_b$ almost
	everywhere on $I$, and $\Psi\colon H^1(I;\R^n)\mto\R^m$ assigns to each $x\in H^1(I;\R^n)$
	the solution set of \eqref{eq:finite_dimensional_lower_level} for the fixed
	parameter $\xi:=x(T)$. The controls in \eqref{eq:ODE_control_terminal_penalty_cost} 
	are chosen from $L^2(I;\R^p)$.
	
	Problem \eqref{eq:ODE_control_terminal_penalty_cost} describes the situation where
	an ODE system has to be controlled in such a way that certain penalty cost resulting
	from the terminal value of the state function as well as the distance to a desirable
	state are minimized with minimal control effort. 
	Optimization problems of this kind arise in the context of gas balancing in energy networks
	and were studied in \cite{BenitaDempeMehlitz2016,BenitaMehlitz2016,KalashnikovBenitaMehlitz2015}.	
	Invoking \Cref{rem:lower_level_abstract_inequality_constraints}, the subsequently stated example,
	and \Cref{thm:existence_abstract}, we obtain the existence of
	an optimal solution associated with \eqref{eq:ODE_control_terminal_penalty_cost}.
\end{example}

Another typical situation arises when the lower level problem \eqref{eq:lower_level_abstract}
is uniquely solvable for each upper level feasible point.
\begin{theorem}\label{thm:existence_abstract_unique_lower_level}
	Assume that there exists a map $\psi\colon X_\textup{ad}\to\ZZ$ sending
	weakly convergent sequences from $X_\textup{ad}$ to weakly convergent sequences in $\ZZ$ such that
	$\Psi(x)=\{\psi(x)\}$ holds for all $x\in X_\textup{ad}$. Then, \eqref{eq:upper_level_abstract}
	possesses an optimal solution.
\end{theorem}
\begin{proof}
	The assumptions of the theorem guarantee that \eqref{eq:upper_level_abstract} is equivalent to
	\begin{align*}
		F(x,\psi(x))&\,\to\,\min\limits_x\\
		x&\,\in\,X_\text{ad}.
	\end{align*}
	Furthermore, $X_\textup{ad}\ni x\mapsto F(x,\psi(x))\in\R$ is weakly sequentially lower semicontinuous on $X_\textup{ad}$
	since $F$ is possesses this property on $\XX\times\ZZ$ while $\psi$ preserves weak convergence
	of sequences from $X_\textup{ad}$.
	Thus, the above problem possesses an optimal solution $\bar x$, i.e.\
	\eqref{eq:upper_level_abstract} possesses the optimal solution $(\bar x,\psi(\bar x))$.
\end{proof}

The above theorem particularly applies to situations where the upper level variable comes from
a finite-dimensional Banach space while the solution operator associated to the lower level
problem is continuous. This setting has been discussed in 
\cite{DempeHarderMehlitzWachsmuth2019,HarderWachsmuth2019} and will be of interest in 
\Cref{sec:stationarity_for_particular_problem}.

\subsection{How to derive necessary optimality conditions in bilevel optimal control}

In order to derive necessary optimality conditions for bilevel programming
problems, one generally aims to transfer the hierarchical model into a 
single-level program first. Therefore, three major approaches are suggested
in the literature. First, whenever the lower level problem possesses a 
uniquely determined solution for each fixed value of the upper level problem,
one could use the associated solution operator to eliminate the lower level
variable from the model. This approach has been used in 
\cite{HarderWachsmuth2019,Mehlitz2017} in order to derive necessary optimality conditions for
bilevel optimal control problems.
Second, it is possible to exploit the optimal value
function from \eqref{eq:value_function} in order to replace \eqref{eq:upper_level_abstract}
equivalently by the so-called \emph{optimal value reformulation}
\begin{align*}
	F(x,z)&\,\to\,\min\limits_{x,z}\\
	x&\,\in\,X_\textup{ad}\\
	f(x,z)-\varphi(x)&\,\leq\,0\\
	z&\,\in\,\Gamma(x).
\end{align*}
In \cite{BenitaDempeMehlitz2016,BenitaMehlitz2016,DempeHarderMehlitzWachsmuth2019,Ye1995,Ye1997},
the authors exploited this idea to infer optimality conditions in the context
of bilevel optimal control. 
We will demonstrate in \Cref{sec:stationarity_for_particular_problem}, how a relaxation
method can be combined with the optimal value approach in order to obtain a satisfactory
stationarity condition for a particular problem class from inverse optimal control.
Finally, as long as the lower level problem is convex w.r.t.\ $z$ and
regular in the sense that a constraint qualification is satisfied at each feasible
point, it is possible to replace the implicit constraint $z\in\Psi(x)$ by 
suitable necessary and sufficient optimality conditions of Karush--Kuhn--Tucker (KKT) type.
In the context of bilevel optimal control, this approach has been discussed in
\cite{MehlitzWachsmuth2016}. In this section, we will briefly sketch this last approach.
Therefore, we have to fix some assumptions first.
\begin{assumption}\label{ass:KKT_ref}
We assume that the mapping $\Gamma$ is given as stated in 
\Cref{rem:lower_level_abstract_inequality_constraints} where $C$ is a cone.
Furthermore, we suppose that $f(x,\cdot)\colon\mathcal Z\to\R$ is convex and that
$g(x,\cdot)\colon\mathcal Z\to\mathcal W$ is $C$-convex for each $x\in X_\text{ad}$.
The latter means that
\[
	\forall z,z'\in\ZZ\,\forall\gamma\in[0,1]\colon\quad
	g(x,\gamma z+(1-\gamma)z')-\gamma g(x,z)-(1-\gamma)g(x,z')\in C
\]
holds true.
\end{assumption}

Du to the postulated assumptions,
for fixed $x\in X_\textup{ad}$, $z\in\Psi(x)$ holds true if and only if
there exists a Lagrange multiplier $\lambda\in\mathcal W^\star$ which solves
the associated lower level KKT system which is given as stated below:
\begin{align*}
	&0=f'_z(x,z)+g'_z(x,z)^\star\lambda,\\
	&\lambda\in C^\circ,\\
	&0=\dual{\lambda}{g(x,z)}{\mathcal W}.
\end{align*}
Here, it was essential that the lower level problem is convex w.r.t.\ $z$ while
Robinson's constraint qualification \eqref{eq:RCQ} is valid at all lower level
feasible points. Due to the above arguments, it is now reasonable to investigate
the so-called \emph{KKT reformulation} associated with \eqref{eq:upper_level_abstract}
which is given as stated below:
\begin{equation}\label{eq:KKT_reformulation_abstract}\tag{KKT}
	\begin{split}
		F(x,z)&\,\to\,\min\limits_{x,z,\lambda}\\
		x&\,\in\,X_\textup{ad}\\
		f'_z(x,z)+g'_z(x,z)^\star\lambda&\,=\,0\\
		g(x,z)&\,\in\,C\\
		\lambda&\,\in\,C^\circ\\
		\dual{\lambda}{g(x,z)}{\mathcal W}&\,=\,0.
	\end{split}
\end{equation}
Let us note that the lower level Lagrange multiplier $\lambda$ plays the role of
a variable in \eqref{eq:KKT_reformulation_abstract}. This may cause that the problems
\eqref{eq:upper_level_abstract} and \eqref{eq:KKT_reformulation_abstract} are not equivalent
w.r.t.\ local minimizers as soon as $\lambda$ is not uniquely determined for
each $x\in X_\textup{ad}$ where $\Psi(x)\neq\varnothing$ holds, see \cite{Mehlitz2017:2}.
As reported in \cite{DempeDutta2012}, this phenomenon is already present in 
standard finite-dimensional bilevel programming.

In the situation where $C=\{0\}$ holds, the final two constraints in \eqref{eq:KKT_reformulation_abstract}
are trivial and can be omitted. Then, \eqref{eq:KKT_reformulation_abstract} reduces to
a standard nonlinear program in Banach spaces which can be tackled via classical arguments.
Related considerations can be found in \cite{HollerKunischBarnard2018}.
The subsequently stated example visualizes this approach.
\begin{example}{Inverse control of Poisson's equation}
	For a bounded domain $\Omega\subset\R^d$ and a parameter vector $x\in\R^n$, we consider
	the parametric optimal control problem
	\begin{align*}
		\tfrac12\norm{y-\mathsmaller\sum\nolimits_{i=1}^nx_if^i}{L^2(\Omega)}^2
			+\tfrac\sigma 2\norm{u}{L^2(\Omega)}^2&\,\to\,\min\limits_{y,u}\\
			-\Delta y&\,=\,u
	\end{align*}
	where $f^1,\ldots ,f^n\in L^2(\Omega)$ are given form functions and $\sigma>0$ is a regularization
	parameter.
	For observations $y_\textup{o},u_\textup{o}\in L^2(\Omega)$ and a nonempty, convex, compact set
	$X_\textup{ad}\subset\R^n$, we consider the superordinate inverse optimal control problem
	\begin{equation}\label{eq:inverse_PDE_control_without_control_constraints}
		\begin{aligned}
		\tfrac12\norm{y-y_\textup{o}}{L^2(\Omega)}^2+\tfrac12\norm{u-u_\textup{o}}{L^2(\Omega)}^2
			&\,\to\,\min\limits_{x,y,u}\\
			x&\,\in\,X_\textup{ad}\\
			(y,u)&\,\in\Psi(x)
		\end{aligned}
	\end{equation}
	where $\Psi\colon\R^n\rightrightarrows H^1_0(\Omega)\times L^2(\Omega)$ represents the
	solution mapping of the aforementioned parametric optimal control problem.
	We can use \Cref{thm:existence_abstract_unique_lower_level} in order to infer the existence of an optimal
	solution associated with this bilevel optimal control problem.
	
	Noting that $-\Delta\colon H^1_0(\Omega)\to H^{-1}(\Omega)$ provides an isomorphism,
	the associated KKT reformulation, given by
	\[
		\begin{split}
		\tfrac12\norm{y-y_\textup{o}}{L^2(\Omega)}^2+\tfrac12\norm{u-u_\textup{o}}{L^2(\Omega)}^2
			&\,\to\,\min\limits_{x,y,u,p}\\
			x&\,\in\,X_\textup{ad}\\
			y-\mathsmaller\sum\nolimits_{i=1}^nx_if^i-\Delta p&\,=\,0\\
			\sigma u-p&\,=\,0\\
			-\Delta y-u&\,=\,0,
		\end{split}
	\]
	is equivalent to the original hierarchical model. 
	One can easily check that Robinson's constraint qualification is valid at
	each feasible point of this program which means that its KKT conditions
	provide a necessary optimality condition for the underlying inverse optimal 
	control problem.
	Thus, whenever $(\bar x,\bar y,\bar u)\in\R^n\times H^1_0(\Omega)\times L^2(\Omega)$
	is a locally optimal solution of \eqref{eq:inverse_PDE_control_without_control_constraints}, 
	then we find multipliers
	$\bar z\in\R^n$, $\bar p,\bar\mu,\bar\rho\in H^1_0(\Omega)$, and $\bar w\in L^2(\Omega)$
	which satisfy
	\[
	\begin{aligned}
		&0\,=\,\bar z-\bigl(\dual{\bar\mu}{f^i}{L^2(\Omega)}\bigr)_{i=1}^n,&\qquad
		&0\,=\,\bar y-y_\textup o+\bar\mu-\Delta\bar\rho,&\\
		&0\,=\,\bar u-u_\textup o+\sigma\bar w-\bar\rho,&\qquad
		&0\,=\,-\Delta\bar\mu-\bar w,&\\
		&\bar z\,\in\,\mathcal N_{X_\textup{ad}}(\bar x),&\qquad
		&0\,=\,\bar y-\mathsmaller\sum\nolimits_{i=1}^n\bar x_if^i-\Delta\bar p,&\\
		&0\,=\,\sigma\bar u-\bar p.&&&
	\end{aligned}
	\]
\end{example}

In case where $C$ is a non-trivial cone, the final three constraints 
of \eqref{eq:KKT_reformulation_abstract}
form a so-called
system of complementarity constraints,
i.e.\ this program is a
\emph{mathematical program with complementarity constrains} (MPCC) in Banach spaces.
As shown in \cite{MehlitzWachsmuth2016}, this results in the violation of Robinson's 
constraint qualification at all feasible points of
\eqref{eq:KKT_reformulation_abstract} and, consequently, the KKT conditions of
\eqref{eq:KKT_reformulation_abstract} may turn out to be too restrictive in order to
yield an applicable necessary optimality condition. Instead, weaker problem-tailored 
stationarity notions and constraint qualifications need to be introduced which respect
the specific variational structure, see 
\cite{Mehlitz2017:2,MehlitzWachsmuth2016,Wachsmuth2015,Wachsmuth2017}.
In bilevel optimal control, complementarity constraints are typically induced by the cone
of nonnegative functions in a suitable function space, e.g.\ $L^2(\Omega)$, 
$H^1_0(\Omega)$, or $H^1(\Omega)$. Respective considerations can be found in
\cite{ClasonDengMehlitzPruefert2019,GuoYe2016,HarderWachsmuth2018a,HarderWachsmuth2018b,
MehlitzWachsmuth2018,MehlitzWachsmuth2019}.

\section{Stationarity conditions in inverse optimal control}
	\label{sec:stationarity_for_particular_problem}

In this section, we demonstrate by means of a specific class of parameter reconstruction
problems how stationarity conditions in bilevel optimal control can be derived.

For a bounded domain $\Omega\subset\R^d$ and a parameter $x\in\R^n_+$, where $\R^n_+$ 
denotes the nonnegative orthant in $\R^n$, we study
the parametric optimal control problem
\begin{equation}\label{eq:lower_level}\tag{P$(x)$}
	\begin{aligned}
		x\cdot j(y)+\tfrac\sigma 2\norm{u}{L^2(\Omega)}^2
			&\,\to\,\min\limits_{y,u}&&&\\
			\opA y-\opB u&\,=\,0&&&\\
			u_a\,\leq\,u&\,\leq\,u_b&&\text{a.e.\ on }\Omega&
	\end{aligned}
\end{equation}
as well as the superordinated bilevel optimal control problem
\begin{equation}\label{eq:upper_level}\tag{IOC}
	\begin{aligned}
		F(x,y,u)&\,\to\,\min\limits_{x,y,u}\\
		x&\,\in\,X_\textup{ad}\\
		(y,u)&\,\in\,\Psi(x)
	\end{aligned}
\end{equation}
where $\Psi\colon\R^n\rightrightarrows\mathcal Y\times L^2(\Omega)$
denotes the solution set mapping of \eqref{eq:lower_level}.
In \eqref{eq:lower_level}, the state equation $\opA y-\opB u=0$ couples
the control $u\in L^2(\Omega)$ and the state $y\in\YY$.
In this regard, $\opA$ can be interpreted
as a differential operator.
Noting that \eqref{eq:upper_level} is motivated by underlying applications from
parameter reconstruction, it is an inverse optimal
control program.

\begin{assumption}
	We assume that $\YY$ and $\WW$ are reflexive Banach spaces.
	The functional $F\colon\R^n\times \YY\times L^2(\Omega)\to\R$ is supposed to be
	continuously Fr\'{e}chet differentiable and convex.
	Let $X_\textup{ad}\subset\R^n_+$ be nonempty and compact.
	The functional $j\colon\YY\to \R^n$ is assumed to be 
	twice continuously Fr\'{e}chet differentiable and its $n$ component functions
	are supposed to be convex.
	Moreover, we assume that the mapping $j$ satisfies $j(\YY)\subset \R^n_+$.
	Furthermore, $\sigma>0$ is fixed. 
	Let linear operators $\opA\in\linop{\YY}{\WW}$ as well as $\opB\in\linop{L^2(\Omega)}{\WW}$
	be chosen such that $\opA$ is continuously invertible while $\opB$ is compact.
	Finally, we assume that $u_a,u_b\colon\Omega\to\overline\R$ are measurable
	functions such that
	\[
		U_\textup{ad}:=\{
			u\in L^2(\Omega)\,|\,
			u_a\leq u\leq u_b\text{ a.e.\ on }\Omega
			\}
	\]
	is nonempty.
\end{assumption}

Below, we present two illustrative examples where all these assumptions hold.

\begin{example}{Weighted lower level target-type objectives}
	We choose $\YY:=H^1_0(\Omega)$, $\WW:=H^{-1}(\Omega)$, as well as $\opA:=-\Delta$ 
	while $\opB$ represents the compact embedding $L^2(\Omega)\hookrightarrow H^{-1}(\Omega)$.
	For fixed functions $y^1_\textup d,\ldots,y^n_\textup{d}\in L^2(\Omega)$,
	the lower level objective function is defined by
	\[
		\R^n\times H^1_0(\Omega)\times L^2(\Omega)\ni
		(x,y,u)
		\mapsto
		\mathsmaller\sum\nolimits_{i=1}^nx_i\norm{y-y^i_\textup d}{L^2(\Omega)}^2
		+
		\tfrac\sigma2\norm{u}{L^2(\Omega)}^2\in\R.
	\]
	The upper level feasible set is given by the standard simplex
	\begin{equation}\label{eq:standard_simplex}
		\{x\in\R^n\,|\,x\geq 0,\,\mathsmaller\sum\nolimits_{i=1}^nx_i=1\}.
	\end{equation}
	Such bilevel optimal control problems, where the precise form of the lower level
	target-type objective mapping has to be reconstructed,
	have been studied in \cite{HarderWachsmuth2019}.	
\end{example}

\begin{example}{Optimal measuring}
	Let $\Omega \subset \R^d$, $d \in \{2,3\}$, be a bounded Lipschitz domain.
	We fix $p \in (3,6)$ as in \cite[Theorem~0.5]{JerisonKenig1995}.
	Let us set $\YY:=W^{1,p}_0(\Omega)$ and $\WW:=W^{-1,p}(\Omega)$.
	Again, we fix $\opA:=-\Delta$,
	and $\opB$ represents the embedding 
	$L^2(\Omega)\hookrightarrow W^{-1,p}(\Omega) := W_0^{1,p'}(\Omega)^\star$
	where $p'=p/(p-1)$ is the conjugate coefficient associated with $p$.
	According to \cite[Theorem~0.5]{JerisonKenig1995}, $\opA$ is continuously invertible.
	Due to the Rellich–Kondrachov theorem,
	the embedding from $W_0^{1,p'}(\Omega)$ to $L^2(\Omega)$
	is compact.
	Since $\opB$ is the adjoint of this embedding,
	Schauder's theorem implies the compactness of $\opB$.
	Furthermore, we note that the embedding 
	$W^{1,p}_0(\Omega)\hookrightarrow C(\overline\Omega)$
	is compact in this setting.	
	
	Let $\omega^1,\ldots,\omega^n\in\overline\Omega$ be fixed points.
	We consider the lower level objective function given by
	\[
		\R^n\times W^{1,p}_0(\Omega)\times L^2(\Omega)\ni
		(x,y,u)\mapsto \mathsmaller\sum\nolimits_{i=1}^nx_i(y(\omega^i)-y_\textup{d}(\omega^i))^2
			+\tfrac\sigma2\norm{u}{L^2(\Omega)}^2\in\R
	\]
	where $y_\textup{d}\in C(\overline\Omega)$ is a given desired state.
	Noting that the state space $W^{1,p}_0(\Omega)$ is continuously embedded into $C(\overline\Omega)$,
	this functional is well-defined.
	At the upper level stage, we minimize
	\[
		\R^n\times W^{1,p}_0(\Omega)\times L^2(\Omega)\ni
		(x,y,u)\mapsto \tfrac12\norm{y-y_\textup d}{L^2(\Omega)}^2+\tfrac12\abs{x}{2}^2
		\in\R
	\]
	where $x$ comes from the standard simplex given in \eqref{eq:standard_simplex}.
	The associated bilevel optimal control problem optimizes the \emph{measurement}
	of the distance between the actual state and the desired state by reduction
	to pointwise evaluations.
\end{example}

\subsection{The lower level problem}\label{sec:lower_level}

For brevity, we denote by $f\colon\R^n\times\YY\times L^2(\Omega)\to\R$ the objective
functional of \eqref{eq:lower_level}. 
By construction, the map $f(x,\cdot,\cdot)\colon\YY\times L^2(\Omega)\to\R$
is convex for each $x\in\R^n_+$.
\begin{lemma}\label{lem:existence_and_uniqueness_lower_level}
	For each $x\in \R^n_+$, \eqref{eq:lower_level} possesses a unique
	optimal solution.
\end{lemma}
\begin{proof}
	Noting that $\opA$ is an isomorphism, we may consider the state-reduced problem
	\begin{equation}\label{eq:state_reduced_lower_level}
		\min\limits_{u}\{f(x,\opS u,u)\,|\,u\in U_\textup{ad}\}
	\end{equation}
	where $\opS:=\opA^{-1}\circ\opB\in\linop{L^2(\Omega)}{\YY}$ is the solution operator
	of the constraining PDE. Due to the above considerations, the linearity of $\opS$,
	and the continuity of all appearing functions,
	the objective functional of \eqref{eq:state_reduced_lower_level} is convex and continuous.
	Observing that $x\cdot j(\opS u)\geq 0$ holds for all $u\in L^2(\Omega)$
	while $L^2(\Omega)\ni u\mapsto\tfrac\sigma 2\norm{u}{L^2(\Omega)}^2\in\R$ is coercive,
	the objective of \eqref{eq:state_reduced_lower_level} is already strongly convex w.r.t.\ $u$. 
	Since $U_\textup{ad}$ is weakly sequentially closed, \eqref{eq:state_reduced_lower_level}
	needs to possess a unique solution $\bar u$. 
	Consequently, $(\opS\bar u,\bar u)$ is the uniquely determined solution of \eqref{eq:lower_level}.	
\end{proof}

Observing that the lower level problem \eqref{eq:lower_level} is regular  
in the sense that Robinson's constraint qualification is valid at all feasible points,
see \Cref{rem:lower_level_abstract_inequality_constraints},
its uniquely determined solution for the fixed parameter $x\in \R^n_+$ is characterized by
the associated KKT system
\begin{subequations}\label{eq:lower_level_KKT}
	\begin{align}
		\label{eq:lower_level_KKT_y}
			&0=j'(y)^\star x+\mathtt A^\star p,\\
		\label{eq:lower_level_KKT_u}
			&0=\sigma u-\mathtt B^\star p+\lambda,\\
		\label{eq:lower_level_normal_cone}
			&\lambda\in\mathcal N_{U_\textup{ad}}(u)
	\end{align}
\end{subequations}
where $p\in\WW^\star$ and $\lambda\in L^2(\Omega)$ are the Lagrange multipliers.

The finding of \cref{lem:existence_and_uniqueness_lower_level} allows us to introduce mappings 
$\psi^y\colon \R^n_+\to\YY$ and $\psi^u\colon \R^n_+\to L^2(\Omega)$
by $\Psi(x)=\{(\psi^y(x),\psi^u(x))\}$ for all $x\in\R^n_+$. 
Since $\opA^\star$ is continuously invertible,
$p$ is uniquely determined by \eqref{eq:lower_level_KKT_y}
and, consequently, the uniqueness of $\lambda$ follows from
\eqref{eq:lower_level_KKT_u}.
This gives rise to the mappings
$\phi^p \colon \R^n_+ \to \WW^\star$
and
$\phi^\lambda \colon \R^n_+ \to L^2(\Omega)$
that assign to each $x\in \R^n_+$ the lower level Lagrange multipliers $p$ and $\lambda$
which characterize the unique minimizer $(\psi^y(x),\psi^u(x))$, respectively.

Before we continue,
we give an auxiliary result on $j$.
\begin{lemma}
	\label{lem:bound_ddj}
	Let $\hat X \subset \R^n_+$ be compact.
	Then, there exists a constant $C > 0$, such that
	\begin{align*}
		\abs{j(y_2) - j(y_1) - j'(y_1)(y_2-y_1)}{2}
		&\le \tfrac C2 \norm{y_2 - y_1}{\YY}^2
		,
		\\
		\norm{j'(y_2) - j'(y_1)}{\linop{\YY}{\R^n}}
		&\le C \norm{y_2 - y_1}{\YY}
	\end{align*}
	with $y_i := \psi^y(x_i)$, $i = 1,2$,
	holds for all $x_1,x_2 \in \hat X$.
\end{lemma}
\begin{proof}
	Since $\hat X$ is assumed to be compact,
	the points $y_2 + (1-t) \, y_1$, $t \in [0,1]$,
	belong to the compact set
	$\hat\YY := \cl\conv\{\psi^y(\hat x) \,|\, \hat x \in \hat X\}$.
	Since $j''$ is continuous, we have
	$C := \sup_{\hat y \in \hat\YY} \norm{j''(\hat y)}{} < \infty$.
	Now, the Taylor estimate
	follows from \cite[Theorem~5.6.1]{Cartan1967}
	and the Lipschitz estimate on $j'$ is clear.
\end{proof}

Below, we want to study the continuity properties of 
the mappings $\psi^y$ and $\psi^u$ as well as $\phi^p$ and $\phi^\lambda$.
\begin{lemma}\label{lem:continuity_of_solution_operator}
	There are continuous functions $C_y,C_u\colon \R^n_+\to\R$ such that the following
	estimates hold:
	\begin{align*}
		\forall x_1,x_2\in \R^n_+\colon\qquad
			\norm{\psi^y(x_1)-\psi^y(x_2)}{\mathcal Y}&\leq C_y(x_1)\abs{x_1-x_2}{2},\\
			\norm{\psi^u(x_1)-\psi^u(x_2)}{L^2(\Omega)}&\leq C_u(x_1)\abs{x_1-x_2}{2}.
	\end{align*}
	Particularly, $\psi^y$ and $\psi^u$ are Lipschitz continuous on $X_\textup{ad}$.
	Additionally, $\phi^p$ and $\phi^\lambda$ are continuous on $\R^n_+$ 
	and Lipschitz continuous on $X_\textup{ad}$.
\end{lemma}
\begin{proof}
	Fix $x_1,x_2\in \R^n_+$ arbitrarily and set $y_i:=\psi^y(x_i)$ as well as $u_i:=\psi^u(x_i)$
	for $i=1,2$. Furthermore, let $p_i\in \WW^\star$ and $\lambda_i\in L^2(\Omega)$ be
	the multipliers which solve \eqref{eq:lower_level_KKT} for $i=1,2$. 
	Testing the associated condition \eqref{eq:lower_level_KKT_u}
	with $u_2-u_1$ and exploiting \eqref{eq:lower_level_normal_cone}, we have
	\begin{align*}
		\dual{\sigma u_1-\mathtt B^\star p_1}{u_2-u_1}{L^2(\Omega)}
			=\dual{-\lambda_1}{u_2-u_1}{L^2(\Omega)}&\geq 0,\\
		\dual{\sigma u_2-\mathtt B^\star p_2}{u_1-u_2}{L^2(\Omega)}
			=\dual{-\lambda_2}{u_1-u_2}{L^2(\Omega)}&\geq 0.
	\end{align*}
	Adding up these inequalities yields
	\[
		\dual{\sigma(u_1-u_2)-\mathtt B^\star(p_1-p_2)}{u_2-u_1}{L^2(\Omega)}\geq 0.
	\]
	Next, we rearrange this inequality and exploit 
	\eqref{eq:lower_level_KKT_y}, $y_i=(\mathtt A^{-1}\circ\mathtt B)u_i$, $i=1,2$,
	 as well as the convexity of the mapping
	$\YY\ni y\mapsto x_2\cdot j(y)\in\R$ in order to obtain
	\begin{align*}
		\sigma\norm{u_1-u_2}{L^2(\Omega)}^2
			&\leq
		\dual{\mathtt B^\star(p_1-p_2)}{u_1-u_2}{L^2(\Omega)}\\
			&=
		\dual{p_1-p_2}{\mathtt B(u_1-u_2)}{\WW}
			=
		\dual{p_1-p_2}{\mathtt A(y_1-y_2)}{\WW}\\
			&=
		\dual{\opA^\star(p_1-p_2)}{y_1-y_2}{\YY}
			=
		\dual{j'(y_2)^\star x_2-j'(y_1)^\star x_1}{y_1-y_2}{\mathcal Y}\\
			&=
		\dual{j'(y_1)^\star(x_2-x_1)}{y_1-y_2}{\mathcal Y}\\
			&\qquad
				-\dual{(j'(y_1)-j'(y_2))^\star x_2}{y_1-y_2}{\mathcal Y}\\
			&\leq
		\dual{j'(y_1)^\star(x_2-x_1)}{y_1-y_2}{\mathcal Y}\\
			&=
		\dual{j'(y_1)^\star(x_2-x_1)}{(\mathtt A^{-1}\circ\mathtt B)(u_1-u_2)}{\YY}\\
			&\leq
		C\norm{j'(y_1)}{\linop{\YY}{\R^n}}\abs{x_1-x_2}{2}\norm{u_1-u_2}{L^2(\Omega)}
	\end{align*}
	for some constant $C>0$ which does not depend on $x_i$, $y_i$, and $u_i$, $i=1,2$.
	This way, we have
	\[
		\norm{u_1-u_2}{L^2(\Omega)}
			\leq
		(C/\sigma)\norm{j'(\psi^y(x_1))}{\linop{\YY}{\R^n}}\abs{x_1-x_2}{2}
	\]
	which yields the estimate
	\[
		\forall x_1,x_2\in \R^n_+\colon\quad
			\norm{\psi^u(x_1)-\psi^u(x_2)}{L^2(\Omega)}
				\leq (C/\sigma)\norm{j'(\psi^y(x_1))}{\linop{\YY}{\R^n}}\abs{x_1-x_2}{2}.
	\]
	As a consequence, the map $\psi^u$ is continuous everywhere on $\R^n_+$. 
	Observing that $\psi^y=\mathtt A^{-1}\circ\mathtt B\circ\psi^u$ holds, $\psi^y$
	is continuous on $\R^n_+$ as well. Recalling that $j$ is continuously Fr\'{e}chet
	differentiable, the desired estimates follow by setting
	\begin{align*}
		\forall x\in \R^n_+\colon\quad
		C_u(x)&:=(C/\sigma)\norm{j'(\psi^y(x))}{\linop{\YY}{\R^n}},\\
		C_y(x)&:=\norm{\mathtt A^{-1}\circ\mathtt B}{\linop{L^2(\Omega)}{\mathcal Y}}\,C_u(x).
	\end{align*}
	This completes the proof for $\psi^y$ and $\psi^u$.

	The continuity of $\phi^p$ and $\phi^\lambda$ on $\R^n_+$ follows easily by continuity
	of $\psi^y$ and $\psi^u$ exploiting \eqref{eq:lower_level_KKT_y}, \eqref{eq:lower_level_KKT_u},
	and the continuity of $j'$.
	Since the map $j' \circ \psi^y \colon \R^n_+ \to \linop{\YY}{\R^n}$ is
	continuous on $\R^n_+$
	and, by \cref{lem:bound_ddj},
	Lipschitz on the compact set $X_\textup{ad}$, we obtain
	\begin{equation*}
		\norm{\phi^p(x_1) - \phi^p(x_2)}{\WW^\star}
		\le
		C \, \norm{j'(\psi^y(x_1))^\star x_1 - j'(\psi^y(x_2))^\star x_2}{\YY}
		\le
		\hat C \, \abs{x_1 - x_2}{2}
	\end{equation*}
	for all $x_1,x_2\in X_\textup{ad}$ and some constants $\hat C,C>0$.
	The Lipschitz continuity of $\phi^\lambda$ on $X_\textup{ad}$ 
	now follows from \eqref{eq:lower_level_KKT_u}.
\end{proof}

We introduce a function $\varphi\colon \R^n_+\to\R$ by means of
\[
	\forall x\in \R^n_+\colon\quad
	\varphi(x):=f(x,\psi^y(x),\psi^u(x)).
\]
Due to the above lemma, $\varphi$ is continuous and equals the optimal value
function associated with \eqref{eq:lower_level}.
Observing that the function $f$ is affine w.r.t.\ the parameter $x$, it is easy
to see that $\varphi$ is concave, see \cite[Proposition~3.5]{FiaccoKyparisis1986} as well.

Next, we are going to study
the differentiability of $\varphi$.
We are facing the problem that $\varphi$
is only defined on the closed set $\R^n_+$.
In fact, for $x \in \R^n \setminus \R^n_+$ the objective function of
\eqref{eq:lower_level} might be non-convex or unbounded from below and
\eqref{eq:lower_level} might fail to possess a minimizer.
To circumvent this difficulty, we first prove
that $\varphi$ admits a first-order Taylor expansion,
and then we extend $\varphi$ to a continuously differentiable function
via Whitney's extension theorem.
\begin{lemma}\label{lem:differentiability_of_phi}
We define the function $\varphi' \colon \R^n_+ \to \R^n$ via
	\[
		\forall \bar x \in \R^n_+\colon\qquad
		\varphi'(\bar x):=j(\psi^y(\bar x)).		
	\]
	Then, $\varphi'$ is continuous and
	for every compact subset $\hat X \subset \R^n_+$
	there exists a constant $C > 0$
	such that the
	Taylor-like estimate
	\begin{equation*}
		\forall x,\bar x \in \hat X\colon\quad
		\bigl\lvert
			\varphi(x)
			-
			\varphi(\bar x)
			-
			\varphi'(\bar x) \cdot (x - \bar x)
		\bigr\rvert
		\le
		C \abs{x - \bar x}{2}^2
	\end{equation*}
	holds.
\end{lemma}
\begin{proof}
	The continuity of $\varphi'$ follows from \Cref{lem:continuity_of_solution_operator}.
	Now, let $\hat X \subset\R^n_+$ be compact.
	For arbitrary $\bar x, x \in \hat X$, we
	define $\bar y := \psi^y(\bar x)$, $\bar u := \psi^u(\bar x)$, 
	$y := \psi^y(x)$, and $u := \psi^u(x)$.
	Then, we have
	\begin{align*}
		&\varphi(x)
		-
		\varphi(\bar x)
		-
		\varphi'(\bar x) \cdot (x - \bar x)
		\\&\qquad
		=
		x\cdot j(y)+\tfrac\sigma 2\norm{u}{L^2(\Omega)}^2
		-
		\bar x\cdot j(\bar y)-\tfrac\sigma 2\norm{\bar u}{L^2(\Omega)}^2
		-
		(x - \bar x)\cdot j(\bar y)
		\\&\qquad
		=
		x\cdot(j(y) - j(\bar y))
		+\sigma \dual{u}{u - \bar u}{L^2(\Omega)}
		-\tfrac\sigma 2\norm{u - \bar u}{L^2(\Omega)}^2
		.
	\end{align*}
	Next, we are going to employ the optimality condition \eqref{eq:lower_level_KKT}.
	To this end, we denote the multipliers at the solution $(y, u)$ for the parameter $x$
	by $p \in \WW^\star$ and $\lambda \in L^2(\Omega)$.
	Now, \eqref{eq:lower_level_KKT} implies
	\begin{equation*}
		\sigma \dual{u}{u - \bar u}{L^2(\Omega)}
		=
		\dual{\mathtt B^\star p}{u - \bar u}{L^2(\Omega)}
		-
		\dual{\lambda}{u - \bar u}{L^2(\Omega)}
	\end{equation*}
	and
	\begin{align*}
		\dual{\mathtt B^\star p}{u - \bar u}{L^2(\Omega)}
		&=
		\dual{p}{\mathtt B (u - \bar u)}{\WW}
		=
		\dual{p}{\mathtt A (y - \bar y)}{\WW}
		=
		\dual{\mathtt A^\star p}{y - \bar y}{\mathcal Y}
		\\&
		=
		-
		\dual{j'(y)^\star x}{y - \bar y}{\mathcal Y}
		=
		-
		x\cdot j'(y) (y - \bar y)
		.
	\end{align*}
	If we denote by $\bar \lambda$ the multiplier associated to the parameter $\bar x$,
	\eqref{eq:lower_level_normal_cone} implies
	\begin{equation*}
		0
		\ge
		-\dual{\lambda}{u - \bar u}{L^2(\Omega)}
		\ge
		\dual{\bar\lambda - \lambda}{u - \bar u}{L^2(\Omega)}
		\ge
		-\norm{\lambda - \bar\lambda}{L^2(\Omega)} \, \norm{u - \bar u}{L^2(\Omega)}
		.
	\end{equation*}
	By combining the above estimates, we get
	\begin{align*}
		\bigl\lvert
			\varphi(x)
			-
			\varphi(\bar x)
			-
			\varphi'(\bar x) \cdot (x - \bar x)
		\bigr\rvert
		&\le
		\abs{x}{2}
		\abs{j(y) - j(\bar y) - j'(y) (y - \bar y)}{2}\\
		&\qquad 
		+
		\norm{\lambda - \bar\lambda}{L^2(\Omega)} \, \norm{u - \bar u}{L^2(\Omega)}
		+
		\tfrac\sigma 2\norm{u - \bar u}{L^2(\Omega)}^2
		.
	\end{align*}
	Now, the claim follows from \cref{lem:bound_ddj,lem:continuity_of_solution_operator}.
\end{proof}
Next, we employ Whitney's extension theorem
to extend $\varphi$ to all of $\R^n$.
\begin{lemma}
	\label{lem:whitney}
	There exists a continuously differentiable function $\hat\varphi \colon \R^n \to \R$
	such that
	$\hat\varphi(x) = \varphi(x)$
	and
	$\hat\varphi'(x) = \varphi'(x)$
	for all $x \in X_\textup{ad}$.
	Here, $\varphi'$ is the function defined in \cref{lem:differentiability_of_phi}.
\end{lemma}
\begin{proof}
	In order to apply Whitney's extension theorem, see \cite[Theorem~2.3.6]{Hoermander2003},
	we have to show that the function $\eta \colon X_\textup{ad} \times X_\textup{ad} \to \R$,
	defined via
	\begin{equation*}
		\eta(x,y) = 0
		\quad\text{if } x = y,
		\qquad
		\eta(x,y)
		=
		\frac{\lvert\varphi(x) - \varphi(y) - \varphi'(y)(x-y)\rvert}{\abs{x -y }{\R^n}}
		\quad\text{if } x \ne y
	\end{equation*}
	for $x,y \in X_\textup{ad}$,
	is continuous on $X_\textup{ad} \times X_\textup{ad}$.
	It is clear that this function is continuous at $(x,y)\in X_\textup{ad}\times X_\textup{ad}$ for $x \ne y$.
	Hence, it remains to show
	\begin{equation*}
		\frac{\lvert\varphi(x) - \varphi(y) - \varphi'(y)(x-y)\rvert}{\abs{x -y }{2}}
		\to
		0
		\qquad\text{for } x,y \to a
	\end{equation*}
	for all $a \in X_\textup{ad}$, but
	this follows from \cref{lem:differentiability_of_phi}
	since $X_\textup{ad}$ is compact.
\end{proof}

Note that we extended $\varphi$ from $X_\textup{ad}$ to $\R^n$ in \Cref{lem:whitney}. 
Technically, this means that the extended function $\hat\varphi$ may possess
different values than the origin optimal value function $\varphi$ 
on $\R^n_+\setminus X_\textup{ad}$.
This, however, does not cause any trouble in the subsequent considerations since
we focus on parameters from $X_\textup{ad}$ only.

\subsection{The optimal value reformulation and its relaxation}\label{sec:upper_level}

Based on \Cref{lem:continuity_of_solution_operator}, the following result follows from 
\Cref{thm:existence_abstract_unique_lower_level} while noting that the upper level
variables are chosen from a finite-dimensional Banach space.
\begin{theorem}\label{thm:existence}
	Problem \eqref{eq:upper_level} possesses an optimal solution.
\end{theorem}

Our aim is to characterize the local minimizers of \eqref{eq:upper_level} by means of necessary
optimality conditions. In order to do so, we want to exploit the continuously differentiable
extension $\hat \varphi\colon\R^n\to\R$ of the optimal value function $\varphi$ associated 
with \eqref{eq:lower_level}, see \Cref{lem:whitney}.
Observing that $\hat\varphi(x)=\varphi(x)$ holds true for all $x\in X_\textup{ad}$,
\eqref{eq:upper_level} can be transferred into the equivalent problem
\begin{equation}\label{eq:OVR}\tag{OVR}
	\begin{split}
		F(x,y,u)&\,\to\,\min\limits_{x,y,u}\\
		x&\,\in\,X_\textup{ad}\\
		f(x,y,u)-\hat \varphi(x)&\,\leq\,0\\
		\opA y-\opB u&\,=\,0\\
		u&\,\in\,U_\textup{ad}
	\end{split}
\end{equation}
where $f$ still denotes the objective of the lower level problem \eqref{eq:lower_level}.
This is a single-level optimization
problem with continuously Fr\'{e}chet differentiable data functions, see \Cref{lem:whitney}.
However, it is easy to check that Robinson's constraint qualification does not hold at the feasible points of \eqref{eq:OVR},
see e.g.\ \cite[Lemma~5.1]{DempeHarderMehlitzWachsmuth2019}. This failure is mainly caused by the fact that
$f(x,y,u)-\hat \varphi(x)\leq 0$ is in fact an equality constraint by definition of the optimal value
function. Due to this lack of regularity, one cannot expect that the classical KKT
conditions provide a necessary optimality condition for \eqref{eq:OVR}.
Furthermore, the nonlinearity of $f$ provokes that the smooth mapping 
$\R^n\times\YY\times L^2(\Omega)\ni(x,y,u)\mapsto f(x,y,u)-\hat \varphi(x)\in\R$ may not serve as
an exact penalty function around local minimizers of \eqref{eq:OVR}.
Thus, approaches related to partial penalization w.r.t.\ the constraint $f(x,y,u)-\hat \varphi(x)\leq 0$,
see e.g.\ \cite{BenitaDempeMehlitz2016,BenitaMehlitz2016,Ye1995,Ye1997}, do not seem to be promising here.

In order to overcome these difficulties, we are going to relax this critical constraint.
More precisely, for a sequence $\{\varepsilon_k\}_{k\in\N}\subset\R$ of positive relaxation
parameters satisfying $\varepsilon_k\downarrow 0$, we investigate the programs
\begin{equation}\label{eq:OVR_relaxed}\tag{OVR$(\varepsilon_k)$}
	\begin{split}
		F(x,y,u)&\,\to\,\min\limits_{x,y,u}\\
		x&\,\in\,X_\textup{ad}\\
		f(x,y,u)-\hat\varphi(x)&\,\leq\,\varepsilon_k\\
		\opA y-\opB u&\,=\,0\\
		u&\,\in\,U_\textup{ad}.
	\end{split}
\end{equation}
One can easily check that this relaxation provokes regularity of all feasible points.
A formal proof of this result parallels the one of \cite[Lemma~5.2]{DempeHarderMehlitzWachsmuth2019}.

We first want to state an existence result for \eqref{eq:OVR_relaxed}.
\begin{lemma}\label{lem:existence_relaxation}
	For each $k\in\N$, \eqref{eq:OVR_relaxed} possesses an optimal solution.
\end{lemma}
\begin{proof}
	Let $\{(x_l,y_l,u_l)\}_{l\in\N}\subset\R^n\times\YY\times L^2(\Omega)$ be a 
	minimizing sequence of \eqref{eq:OVR_relaxed}, i.e.\ a sequence of feasible points
	whose associated objective values tend to the infimal value $\alpha\in\overline\R$ of
	\eqref{eq:OVR_relaxed}. The compactness of $X_\textup{ad}$ 
	implies that $\{x_l\}_{l\in\N}$ is bounded
	and, thus, converges along a subsequence (without relabeling) 
	to $\bar x\in X_\textup{ad}$. By feasibility, we have
	\[
		\tfrac\sigma2\norm{u_l}{L^2(\Omega)}^2
		\leq
		x_l\cdot j(y_l)+\tfrac\sigma2\norm{u_l}{L^2(\Omega)}^2
		\leq
		\varepsilon_k+\hat\varphi(x_l)
	\]
	for each $l\in\N$. By boundedness of $\{x_l\}_{l\in\N}$ and 
	continuity of $\hat \varphi$ on $X_\textup{ad}$,	
	we obtain boundedness of
	$\{u_l\}_{l\in\N}$. Consequently, the latter converges weakly 
	(along a subsequence without relabeling) to	$\bar u\in L^2(\Omega)$ 
	which belongs to the weakly sequentially closed set $U_\textup{ad}$.
	Since $\opB$ is compact, we have $y_l\to\bar y$ in $\mathcal Y$
	by validity of the state equation.
	Here, we used $\bar y:=(\opA^{-1}\circ\opB)\bar u$.
	
	Recall that $j$ and $\hat\varphi$ are continuous functions.
	Exploiting the weak sequential
	lower semicontinuity of (squared) norms, we obtain
	\begin{align*}
		f(\bar x,\bar y,\bar u)-\hat\varphi(\bar x)
		&\leq
		\lim\limits_{l\to\infty}x_l\cdot j(y_l)
			+\liminf\limits_{l\to\infty}\tfrac\sigma2\norm{u_l}{L^2(\Omega)}^2
			+\lim\limits_{l\to\infty}(-\hat\varphi)(x_l)\\
		&=
			\liminf\limits_{l\to\infty}
				\left(	x_l\cdot j(y_l)
						+\tfrac\sigma2\norm{u_l}{L^2(\Omega)}^2
						-\hat \varphi(x_l)
				\right)
		\leq
			\varepsilon_k,
	\end{align*}
	i.e.\ $(\bar x,\bar y,\bar u)$ is feasible to \eqref{eq:OVR_relaxed}.
	Finally, the weak sequential lower semicontinuity of $F$ yields
	\[
		F(\bar x,\bar y,\bar u)
		\leq
		\liminf\limits_{l\to\infty}F(x_l,y_l,u_l)
		\leq\alpha,
	\]
	i.e.\ $(\bar x,\bar y,\bar u)$ is a global minimizer of \eqref{eq:OVR_relaxed}.
\end{proof}

Next, we investigate the behavior of a sequence $\{(\bar x_k,\bar y_k,\bar u_k)\}_{k\in\N}$
of global minimizers associated with \eqref{eq:OVR_relaxed} as $k\to\infty$.
\begin{theorem}\label{thm:relaxation_behavior}
	For each $k\in\N$, let $(\bar x_k,\bar y_k,\bar u_k)\in \R^n\times\YY\times L^2(\Omega)$
	be a global minimizer of \eqref{eq:OVR_relaxed}.
	Then, the sequence $\{(\bar x_k,\bar y_k,\bar u_k)\}_{k\in\N}$ possesses a subsequence (without
	relabeling) such that the convergences $\bar x_k\to\bar x$, $\bar y_k\to\bar y$, and $\bar u_k\to\bar u$
	hold where $(\bar x,\bar y,\bar u)$ is a global minimizer of \eqref{eq:OVR} and, thus,
	of \eqref{eq:upper_level}.
\end{theorem}
\begin{proof}
	Due to compactness of $X_\textup{ad}$, 
	the sequence $\{\bar x_k\}_{k\in\N}$ is bounded and converges 
	along a subsequence (without relabeling) to some $\bar x\in X_\textup{ad}$.
	We set $\bar y:=\psi^y(\bar x)$ and $\bar u:=\psi^u(\bar x)$.
	
	Let us set $y_k:=\psi^y(\bar x_k)$ and $u_k:=\psi^u(\bar x_k)$ for each $k\in\N$.
	Due to the componentwise convexity and differentiability of the mapping $j$, 
	we obtain
	\begin{align*}
		\tfrac\sigma2\norm{\bar u_k-u_k}{L^2(\Omega)}^2
		&=
		\tfrac\sigma2\norm{\bar u_k}{L^2(\Omega)}^2
			-\dual{\bar u_k-u_k}{\sigma u_k}{L^2(\Omega)}
			-\tfrac\sigma 2\norm{u_k}{L^2(\Omega)}^2\\
		&\leq
			\bar x_k\cdot(j(\bar y_k)-j(y_k)
			-j'(y_k)(\bar y_k-y_k))\\
		&\qquad
			+ \tfrac\sigma2\norm{\bar u_k}{L^2(\Omega)}^2
			-\dual{\bar u_k-u_k}{\sigma u_k}{L^2(\Omega)}
			-\tfrac\sigma 2\norm{u_k}{L^2(\Omega)}^2\\
		&=
			f(\bar x_k,\bar y_k,\bar u_k)-\hat\varphi(\bar x_k)\\
		&\qquad
			-\bar x_k\cdot j'(y_k)(\bar y_k-y_k)
			-\dual{\bar u_k-u_k}{\sigma u_k}{L^2(\Omega)}\\
		&\leq
			f(\bar x_k,\bar y_k,\bar u_k)-\hat \varphi(\bar x_k)\leq\varepsilon_k.			
	\end{align*}
	Here, we used feasibility of $(\bar y_k,\bar u_k)$ and optimality
	of $(y_k,u_k)$ for \eqref{eq:lower_level} where $x:=\bar x_k$ holds.
	
	The above considerations as well as the continuity of $\psi^u$ yield
	\begin{align*}
			0
		&\leq
			\lim\limits_{k\to\infty}\norm{\bar u_k-\bar u}{L^2(\Omega)}
		\leq
			\lim\limits_{k\to\infty}\left(\norm{\bar u_k-u_k}{L^2(\Omega)}
				+\norm{u_k-\bar u}{L^2(\Omega)}\right)\\
		&=
			\lim\limits_{k\to\infty}\left(\norm{\bar u_k-u_k}{L^2(\Omega)}
				+\norm{\psi^u(\bar x_k)-\psi^u(\bar x)}{L^2(\Omega)}\right)\\
		&\leq
			\lim\limits_{k\to\infty}\left(\sqrt{2\varepsilon_k/\sigma}+\norm{\psi^u(\bar x_k)-\psi^u(\bar x)}{L^2(\Omega)}\right)
		=
			0,
	\end{align*}
	i.e.\ $\bar u_k\to\bar u$ follows. Due to $\bar y_k=(\opA^{-1}\circ\opB)\bar u_k$ and
	$\bar y=(\opA^{-1}\circ\opB)\bar u$, we also have $\bar y_k\to\bar y$.
	
	Since each feasible point $(x,y,u)\in\R^n\times\YY\times L^2(\Omega)$ of \eqref{eq:OVR}
	is a feasible point of \eqref{eq:OVR_relaxed} for arbitrary $k\in\N$, we have
	$F(\bar x_k,\bar y_k,\bar u_k)\leq F(x,y,u)$ for all $k\in\N$. Taking the limit $k\to\infty$
	while observing that $F$ is (weakly sequentially) lower semicontinuous, we have
	$F(\bar x,\bar y,\bar u)\leq F(x,y,u)$, i.e.\ $(\bar x,\bar y,\bar u)$ is a global minimizer
	of \eqref{eq:OVR} and, thus, of \eqref{eq:upper_level}.
\end{proof}

Clearly, the above theorem is of limited use for the numerical treatment of \eqref{eq:upper_level}
since the programs \eqref{eq:OVR_relaxed} are nonconvex optimal control problems whose
constraints comprise the implicitly known function $\hat\varphi$. 
However, we can exploit \Cref{thm:relaxation_behavior} for the derivation of a necessary optimality
condition for \eqref{eq:upper_level}.

\subsection{Derivation of stationarity conditions}

For the derivation of a necessary optimality condition which characterizes the local minimizers
of \eqref{eq:upper_level}, we will exploit the relaxation approach described in \Cref{sec:upper_level}.
Combining the KKT systems of \eqref{eq:OVR_relaxed} and \eqref{eq:lower_level} will lead
to a stationarity system for global minimizers of \eqref{eq:upper_level}.
Afterwards, this result can be extended to all local minimizers of \eqref{eq:upper_level}. 

As already mentioned in \cref{sec:upper_level}, we cannot rely on the KKT conditions of \eqref{eq:OVR} to be applicable
necessary optimality conditions for \eqref{eq:upper_level}.
In order to derive a \emph{reasonable} stationarity system, we first observe that for given $x\in X_\textup{ad}$,
we can characterize $(\psi^y(x),\psi^u(x))$ to be the uniquely determined solution of the KKT system
\eqref{eq:lower_level_KKT}
associated with \eqref{eq:lower_level}. Plugging this system into the constraints of 
\eqref{eq:upper_level} in order to eliminate the implicit constraint $(y,u)\in\Psi(x)$, we arrive at
the associated KKT reformulation
\begin{equation}\label{eq:KKT_reformulation}
	\begin{split}
	F(x,y,u)&\,\to\,\min\limits_{x,y,u,p,\lambda}\\
		x&\,\in\,X_\textup{ad}\\
		\opA y-\opB u&\,=\,0\\
		j'(y)^\star x+\opA^\star p&\,=\,0\\
		\sigma u-\opB^\star p+\lambda&\,=\,0\\
		(u,\lambda)&\,\in\,\gph\mathcal N_{U_\textup{ad}}
	\end{split}
\end{equation}
where $\mathcal N_{U_\textup{ad}}\colon L^2(\Omega)\rightrightarrows L^2(\Omega)$ denotes
the normal cone mapping associated with $U_\textup{ad}$. A simple calculation shows
\[
	\gph\mathcal N_{U_\textup{ad}}
	=
	\left\{
		(u,\lambda)\in U_\textup{ad}\times L^2(\Omega)\,\middle|\,
		\begin{aligned}
			&\lambda\geq 0\quad\text{a.e.\ on }\{\omega\in\Omega\,|\,u(\omega)>u_a(\omega)\}\\
			&\lambda\leq 0\quad\text{a.e.\ on }\{\omega\in\Omega\,|\,u(\omega)<u_b(\omega)\}
		\end{aligned}
	\right\}.
\]
In order to infer a stationarity system for \eqref{eq:upper_level}, we compute the roots
of the partial derivatives of the MPCC-Lagrangian associated with \eqref{eq:KKT_reformulation}.
The properties of the multipliers which address the equilibrium condition 
$(u,\lambda)\in\gph\mathcal N_{U_\textup{ad}}$ are motivated by the pointwise structure
of this set and the theory on finite-dimensional complementarity problems.
\begin{definition}\label{def:C_stationarity}
	A feasible point $(\bar x,\bar y,\bar u)\in \R^n\times\YY\times L^2(\Omega)$ of \eqref{eq:upper_level}
	is said to be weakly stationary (W-stationary) whenever there exist multipliers $\bar z\in\R^n$,
	$\bar\mu\in\YY$, $\bar p,\bar\rho\in\WW^\star$, and $\bar\lambda,\bar w,\bar\xi\in L^2(\Omega)$
	which satisfy
	\begin{subequations}
		\begin{align}
		\label{eq:CSt_x}
			0&\,=\,F'_x(\bar x,\bar y,\bar u)+\bar z
				+j'(\bar y)\bar\mu,\\
		\label{eq:CSt_y}
			0&\,=\,F'_y(\bar x,\bar y,\bar u)+\opA^\star\bar\rho
				+j''(\bar y)(\bar\mu)^\star\bar x,\\
		\label{eq:CSt_u}
			0&\,=\,F'_u(\bar x,\bar y,\bar u)+\sigma\bar w-\opB^\star\bar\rho+\bar\xi,\\
		\label{eq:CSt_p}
			0&\,=\,\opA\bar\mu-\opB\bar w,\\
		\label{eq:CSt_z}
			\bar z&\,\in\,\mathcal N_{X_\textup{ad}}(\bar x),\\
		\label{eq:CSt_ll_y}
			0&\,=\,j'(\bar y)^\star\bar x+\opA^\star \bar p\\
		\label{eq:CSt_ll_u}
			0&\,=\,\sigma \bar u-\opB^\star \bar p+\bar \lambda\\
		\label{eq:CSt_ll_lambda_a}
			\bar\lambda&\,\geq\,0\quad\text{a.e.\ on }I^{a+}(\bar u),\\
		\label{eq:CSt_ll_lambda_b}
			\bar\lambda&\,\leq\,0\quad\text{a.e.\ on }I^{b-}(\bar u),\\
		\label{eq:CSt_xi}
			\bar\xi&\,=\,0\quad\text{a.e.\ on }I^{a+}(\bar u)\cap I^{b-}(\bar u),\\
		\label{eq:CSt_w}
			\bar w&\,=\,0\quad\text{a.e.\ on }\{\omega\in\Omega\,|\,\bar\lambda(\omega)\neq 0\}.
		\end{align}
	\end{subequations}
	Whenever these multipliers additionally satisfy
	\begin{equation}
		\label{eq:CSt_clarke}
			\bar\xi\bar w\,\geq\,0\quad\text{a.e.\ on }\Omega,
	\end{equation}
	$(\bar x,\bar y,\bar u)$ is called Clarke-stationary (C-stationary).
	If \eqref{eq:CSt_clarke} can be strengthened to
	\[
		\begin{aligned}
			&\bar\xi\,\leq\,0\,\land\,\bar w\,\leq\,0&
				&\text{a.e.\ on }\{\omega\in\Omega\,|\,\bar\lambda(\omega)=0\,
									\land\,\bar u(\omega)=u_a(\omega)\},&\\
			&\bar\xi\,\geq\,0\,\land\,\bar w\,\geq\,0&
				&\text{a.e.\ on }\{\omega\in\Omega\,|\,\bar\lambda(\omega)=0\,
									\land\,\bar u(\omega)=u_b(\omega)\},&
		\end{aligned}
	\]
	then $(\bar x,\bar y,\bar u)$ is referred to as strongly stationary (S-stationary).
	Here, the measurable sets $I^{a+}(\bar u)$ and $I^{b-}(\bar u)$ are given by
	\[
		I^{a+}(\bar u):=\{\omega\in\Omega\,|\,\bar u(\omega)>u_a(\omega)\},
		\qquad
		I^{b-}(\bar u):=\{\omega\in\Omega\,|\,\bar u(\omega)<u_b(\omega)\}.
	\]
	Note that all subsets of $\Omega$ appearing above are well-defined 
	up to subsets of $\Omega$ possessing measure zero.
\end{definition}

Observe that the conditions \eqref{eq:CSt_ll_y} - \eqref{eq:CSt_ll_lambda_b} just provide the
KKT system \eqref{eq:lower_level_KKT} of \eqref{eq:lower_level} for $x:=\bar x$ which characterizes
the associated lower level Lagrange multipliers $\bar p$ and $\bar\lambda$. This way, a feasible
point of \eqref{eq:KKT_reformulation} is fixed and the actual respective 
W-, C-, and S-stationarity conditions can be inferred.

In line with the results from \cite{DempeHarderMehlitzWachsmuth2019,HarderWachsmuth2019}, we are
going to show that the local minimizers of \eqref{eq:upper_level} are C-stationary.
In order to do that, we choose an arbitrary sequence $\{\varepsilon_k\}_{k\in\N}\subset\R$
of positive penalty parameters tending to zero as $k\to\infty$. Due to \Cref{lem:existence_relaxation},
the program \eqref{eq:OVR_relaxed} possesses a global minimizer 
$(\bar x_k,\bar y_k,\bar u_k)\in\R^n\times\YY\times L^2(\Omega)$. As we mentioned in \Cref{sec:upper_level},
\eqref{eq:OVR_relaxed} is regular as well as smooth at this point and, thus, we find multipliers $z_k\in\R^n$,
$\alpha_k\in\R$, $p_k\in\WW^\star$, and $\lambda_k\in L^2(\Omega)$ which solve the associated KKT system
\begin{subequations}
		\begin{align}
		\label{eq:KKT_relaxed_x}
			0&\,=\,F'_x(\bar x_k,\bar y_k,\bar u_k)+z_k
				+\alpha_k(j(\bar y_k)-\hat\varphi'(\bar x_k)),\\
		\label{eq:KKT_relaxed_y}
			0&\,=\,F'_y(\bar x_k,\bar y_k,\bar u_k)+\alpha_k j'(\bar y_k)^\star\bar x_k+\opA^\star p_k,\\
		\label{eq:KKT_relaxed_u}
			0&\,=\,F'_u(\bar x_k,\bar y_k,\bar u_k)
				+\alpha_k\sigma\bar u_k-\opB^\star p_k+\lambda_k,\\
		\label{eq:KKT_relaxed_z}
			z_k&\,\in\,\mathcal N_{X_\textup{ad}}(\bar x_k),\\
		\label{eq:KKT_relaxed_alpha}
			0&\,\leq\,\alpha_k\,\perp\,f(\bar x_k,\bar y_k,\bar u_k)-\hat \varphi(\bar x_k)-\varepsilon_k,\\
		\label{eq:KKT_relaxed_lambda}
			\lambda_k&\,\in\,\mathcal N_{U_\textup{ad}}(\bar u_k).
		\end{align}
\end{subequations}
Furthermore, an evaluation of the lower level KKT system \eqref{eq:lower_level_KKT} yields
\begin{subequations}
	\begin{align}
		\label{eq:KKT_ll_y}
			0&\,=\,j'(\psi^y(\bar x_k))^\star\bar x_k+\opA^\star\phi^p(\bar x_k),\\
		\label{eq:KKT_ll_u}
			0&\,=\,\sigma\psi^u(\bar x_k)-\opB^\star\phi^p(\bar x_k)+\phi^\lambda(\bar x_k),\\
		\label{eq:KKT_ll_lambda}
			\phi^\lambda(\bar x_k)&\,\in\,\mathcal N_{U_\textup{ad}}(\psi^u(\bar x_k)).
	\end{align}
\end{subequations}
Recall that $\phi^p\colon \R^n_+\to\WW^\star$ and $\phi^\lambda\colon \R^n_+\to L^2(\Omega)$ 
denote the Lagrange multiplier mappings
associated with the lower level problem \eqref{eq:lower_level} which are continuous due to 
\Cref{lem:continuity_of_solution_operator}.

Due to \Cref{thm:relaxation_behavior}, we may assume that we have $\bar x_k\to\bar x$, $\bar y_k\to\bar y$,
and $\bar u_k\to\bar u$ where $(\bar x,\bar y,\bar u)\in\R^n\times\YY\times L^2(\Omega)$ is a global minimizer of
\eqref{eq:upper_level}.

Summarizing all these assumptions, we obtain the following results.
\begin{lemma}\label{lem:convergence_of_multipliers}
	There exist $\bar z\in\R^n$, $\bar\mu\in\YY$, $\bar\rho\in\WW^\star$, and $\bar w,\bar\xi\in L^2(\Omega)$
	such that the convergences
	\begin{subequations}
		\begin{align}
			\label{eq:multipliers_z}
				z_k&\,\to\,\bar z\quad\text{in}\,\R^n,\\
			\label{eq:multipliers_mu}
				\alpha_k(\bar y_k-\psi^y(\bar x_k))&\,\to\,\bar\mu\quad\text{in}\,\YY,\\
			\label{eq:multipliers_w}
				\alpha_k(\bar u_k-\psi^u(\bar x_k))&\,\weakly\,\bar w\quad\text{in}\,L^2(\Omega),\\
			\label{eq:multipliers_rho}
				p_k-\alpha_k\phi^p(\bar x_k)&\,\to\,\bar\rho\quad\text{in}\,\WW^\star,\\
			\label{eq:multipliers_xi}
				\lambda_k-\alpha_k\phi^\lambda(\bar x_k)&\,\weakly\,\bar\xi\quad\text{in}\,L^2(\Omega)
		\end{align}
	\end{subequations}
	hold at least along a subsequence. 
	Furthermore, the above limits satisfy the conditions
	\eqref{eq:CSt_x}, \eqref{eq:CSt_y}, \eqref{eq:CSt_u}, \eqref{eq:CSt_p}, and \eqref{eq:CSt_z}.
\end{lemma}
\begin{proof}
	We multiply \eqref{eq:KKT_ll_y} by $\alpha_k$ and subtract the resulting equation
	from \eqref{eq:KKT_relaxed_y} in order to obtain
	\begin{equation}\label{eq:difference_y}
		0\,=\,F'_y(\bar x_k,\bar y_k,\bar u_k)
			+\alpha_k(j'(\bar y_k)-j'(\psi^y(\bar x_k)))^\star\bar x_k
			+\opA^\star(p_k-\alpha_k\phi^p(\bar x_k)).
	\end{equation}
	Testing this equation with $\bar y_k-\psi^y(\bar x_k)$ while noticing that the first-order derivative 
	of a convex function is a monotone operator, we have
	\begin{align*}
		&\dual{F'_y(\bar x_k,\bar y_k,\bar u_k)+\opA^\star(p_k-\alpha_k\phi^p(\bar x_k))}{\bar y_k-\psi^y(\bar x_k)}{\YY}\\
		&\qquad=
			-\alpha_k\dual{(j'(\bar y_k)-j'(\psi^y(\bar x_k)))^\star\bar x_k}{\bar y_k-\psi^y(\bar x_k)}{\YY}
		\leq 0.
	\end{align*}
	This is used to obtain
	\begin{align*}
		&\dual{\opB^\star(p_k-\alpha_k\phi^p(\bar x_k))}{\bar u_k-\psi^u(\bar x_k)}{L^2(\Omega)}\\
		&\qquad=
			\dual{p_k-\alpha_k\phi^p(\bar x_k)}{\opB(\bar u_k-\psi^u(\bar x_k))}{\WW}\\
		&\qquad=
			\dual{p_k-\alpha_k\phi^p(\bar x_k)}{\opA(\bar y_k-\psi^y(\bar x_k))}{\WW}\\
		&\qquad=
			\dual{\opA^\star(p_k-\alpha_k\phi^p(\bar x_k))}{\bar y_k-\psi^y(\bar x_k)}{\YY}\\
		&\qquad\leq
			\dual{-F'_y(\bar x_k,\bar y_k,\bar u_k)}{\bar y_k-\psi^y(\bar x_k)}{\YY}\\
		&\qquad=
			\dual{-F'_y(\bar x_k,\bar y_k,\bar u_k)}{(\opA^{-1}\circ\opB)(\bar u_k-\psi^u(\bar x_k))}{\YY}\\
		&\qquad\leq
			C\norm{\bar u_k-\psi^u(\bar x_k)}{L^2(\Omega)}
	\end{align*}
	for some constant $C>0$ since $\{F'_y(\bar x_k,\bar y_k,\bar u_k)\}_{k\in\N}$ is bounded.
	Next, we multiply \eqref{eq:KKT_ll_u} by $\alpha_k$ and subtract this from \eqref{eq:KKT_relaxed_u}
	in order to obtain
	\begin{equation}\label{eq:difference_u}
	\begin{aligned}
		0&\,=\,F'_u(\bar x_k,\bar y_k,\bar u_k)
		+\alpha_k\sigma(\bar u_k-\psi^u(\bar x_k))\\
		&\qquad\qquad
		-\opB^\star(p_k-\alpha_k\phi^p(\bar x_k))
		+\lambda_k-\alpha_k\phi^\lambda(\bar x_k)	.	
	\end{aligned}
	\end{equation}
	Testing this with $\bar u_k-\psi^u(\bar x_k)$ and exploiting the above estimate as well as
	 the definition of the normal cone, we obtain
	\begin{align*}
		&\alpha_k\sigma\norm{\bar u_k-\psi^u(\bar x_k)}{L^2(\Omega)}^2\\
		&\qquad=
			\dual{-F'_u(\bar x_k,\bar y_k,\bar u_k)+\opB^\star(p_k-\alpha_k\phi^p(\bar x_k))}{\bar u_k-\psi^u(\bar x_k)}{L^2(\Omega)}\\
		&\qquad\qquad
			+\dual{\alpha_k\phi^\lambda(\bar x_k)-\lambda_k}{\bar u_k-\psi^u(\bar x_k)}{L^2(\Omega)}\\
		&\qquad\leq
			\dual{-F'_u(\bar x_k,\bar y_k,\bar u_k)+\opB^\star(p_k-\alpha_k\phi^p(\bar x_k))}{\bar u_k-\psi^u(\bar x_k)}{L^2(\Omega)}\\
		&\qquad\leq
			\hat C\norm{\bar u_k-\psi^u(\bar x_k)}{L^2(\Omega)}		
	\end{align*}
	for a constant $\hat C>0$.
	Consequently, the sequence $\{\alpha_k(\bar u_k-\psi^u(\bar x_k))\}_{k\in\N}$
	is bounded and, therefore, possesses a weakly convergent subsequence (without relabelling)
	whose weak limit will be denoted by $\bar w$. Thus, we have shown \eqref{eq:multipliers_w}.
	Due to the relation $\bar y_k-\psi^y(\bar x_k)=(\opA^{-1}\circ\opB)(\bar u_k-\psi^u(\bar x_k))$
	and the compactness of $\opB$, we
	obtain the strong convergence $\alpha_k(\bar y_k-\psi^y(\bar x_k))\to\bar\mu$ for some
	$\bar\mu\in\YY$ satisfying \eqref{eq:CSt_p}. Thus, we have \eqref{eq:multipliers_mu}.
	
	Since $j$ is assumed to be continuously Fr\'{e}chet differentiable, $j$ is strictly differentiable. 
	Noting that the strong convergences $\bar y_k\to\bar y$ and $\psi^y(\bar x_k)\to\bar y$ hold, 
	we have
	\[
		\frac{j(\bar y_k)-j(\psi^y(\bar x_k))
			-j'(\bar y)(\bar y_k-\psi^y(\bar x_k))}
			{\norm{\bar y_k-\psi^y(\bar x_k)}{\YY}}
		\to 0.
	\]
	Observing that $\{\alpha_k(\bar y_k-\psi^y(\bar x_k))\}_{k\in\N}$ is particularly bounded, we obtain
	\[
		\alpha_k\left(j(\bar y_k)-j(\psi^y(\bar x_k))
			-j'(\bar y)(\bar y_k-\psi^y(\bar x_k))\right)
		\to 0.
	\]
	Since the convergence
	$j'(\bar y)(\alpha_k(\bar y_k-\psi^y(\bar x_k)))\to j'(\bar y)\bar\mu$
	is clear from \eqref{eq:multipliers_mu}, we infer 
	\begin{equation}\label{eq:convergence_first_derivative}
		\alpha_k(j(\bar y_k)-j(\psi^y(\bar x_k)))
		\to
		j'(\bar y)\bar\mu.
	\end{equation}
	Observing that $j$ is twice continuously Fr\'{e}chet differentiable,
	$j'$ is strictly differentiable. Thus, we can reprise the above arguments
	in order to show the convergence
	\begin{equation}\label{eq:convergence_second_oder_derivative_y}
				\alpha_k(j'(\bar y_k)-j'(\psi^y(\bar x_k)))^\star\bar x_k
				\to
				j''(\bar y)(\bar\mu)^\star\bar x.
	\end{equation}
		Next, we combine \eqref{eq:difference_y}, \eqref{eq:convergence_second_oder_derivative_y}, 
		and the fact that $\opA^\star$ is continuously invertible
		in order to obtain \eqref{eq:multipliers_rho} for some $\bar\rho\in\WW^\star$
		(along a subsequence) which satisfies \eqref{eq:CSt_y}.
		Now, we can infer \eqref{eq:multipliers_xi} for some $\bar\xi\in L^2(\Omega)$ 
		from \eqref{eq:difference_u} in a similar way. Taking the weak limit in \eqref{eq:difference_u}
		yields \eqref{eq:CSt_u}. Due to \Cref{lem:differentiability_of_phi,lem:whitney},
		\eqref{eq:KKT_relaxed_x} is equivalent to
		\[
			0\,=\,F'_x(\bar x_k,\bar y_k,\bar u_k)
				+z_k
				+\alpha_k(j(\bar y_k)-j(\psi^y(\bar x_k))).
		\]
		Due to the convergences $F'_x(\bar x_k,\bar y_k,\bar u_k)\to F'_x(\bar x,\bar y,\bar u)$
		and \eqref{eq:convergence_first_derivative}, we infer \eqref{eq:multipliers_z}
		for some $\bar z\in\R^n$ along a subsequence. 
		Particularly, we have \eqref{eq:CSt_x}. Finally, \eqref{eq:CSt_z} follows
		by definition of the normal cone while observing $z_k\to\bar z$ and $\bar x_k\to\bar x$.
		This completes the proof.
\end{proof}

In the subsequent lemma, we characterize the multipliers $\bar w$ and $\bar\xi$ from
\Cref{lem:convergence_of_multipliers} in more detail.
\begin{lemma}\label{lem:CSt}
	Let $\bar w,\bar\xi\in L^2(\Omega)$ be the multipliers characterized in
	\Cref{lem:convergence_of_multipliers}.
	Then, \eqref{eq:CSt_xi}, \eqref{eq:CSt_w}, and \eqref{eq:CSt_clarke} hold.
\end{lemma}
\begin{proof}
	Due to the strong convergence of $\{\bar u_k\}_{k\in\N}$ and $\{\psi^u(\bar x_k)\}_{k\in\N}$
	to $\bar u$ in $L^2(\Omega)$, these convergences hold pointwise almost everywhere on $\Omega$ along
	a subsequence (without relabelling). From $\lambda_k\in\mathcal N_{U_\textup{ad}}(\bar u_k)$
	and $\phi^\lambda(\bar x_k)\in\mathcal N_{U_\textup{ad}}(\psi^u(\bar x_k))$, we have
	\[
		\lambda_k-\alpha_k\phi^\lambda(\bar x_k)=0
		\qquad
		\text{a.e.\ on}
		\qquad
		\left\{\omega\in\Omega\,\middle|\,
			\begin{aligned}
				&u_a(\omega)<\bar u_k(\omega)<u_b(\omega)\\
				&u_a(\omega)<\psi^u(\bar x_k)(\omega)<u_b(\omega)
			\end{aligned}
		\right\}
	\]
	by definition of the normal cone.
	The aforementioned pointwise a.e.\ convergence yields 
	$\lambda_k(\omega)-\alpha_k\phi^\lambda(\bar x_k)(\omega)\to 0$
	for almost every $\omega\in I^{a+}(\bar u)\cap I^{b-}(\bar u)$.
	Since we already have $\lambda_k-\alpha_k\phi^\lambda(\bar x_k)\weakly\bar\xi$
	from \eqref{eq:multipliers_xi}, \eqref{eq:CSt_xi} follows.
	
	Next, we show \eqref{eq:CSt_w}. If $\{\alpha_k\}_{k\in\N}$ is bounded, then
	$\bar w=0$ follows from \eqref{eq:multipliers_w} and \eqref{eq:CSt_w} holds
	trivially. Thus, we assume $\alpha_k\to +\infty$.
	By continuity of $\phi^p$ and $\phi^\lambda$, see \Cref{lem:continuity_of_solution_operator},
	we have $\phi^p(\bar x_k)\to\phi^p(\bar x)$ and
	$\phi^\lambda(\bar x_k)\to\phi^\lambda(\bar x)$.
	Noting that the lower level Lagrange multipliers are uniquely determined while
	observing that $\psi^y(\bar x_k)\to\bar y$ and $\psi^u(\bar x_k)\to\bar u$ hold,
	we have $\phi^p(\bar x_k)\to\bar p$ and $\phi^\lambda(\bar x_k)\to\bar\lambda$.
	Here, $\bar p\in\WW^\star$ and $\bar\lambda\in L^2(\Omega)$ satisfy 
	the conditions \eqref{eq:CSt_ll_y}, \eqref{eq:CSt_ll_u}, \eqref{eq:CSt_ll_lambda_a},
	and \eqref{eq:CSt_ll_lambda_b}.
	Thus, \eqref{eq:multipliers_xi} yields the strong
	convergence $\alpha_k^{-1}\lambda_k\to\bar\lambda$.
	Let $G\subset\Omega$ be measurable and $\chi_G\in L^\infty(\Omega)$ be its characteristic
	function which equals $1$ on $G$ and vanishes on $\Omega\setminus G$.
	By definition of the normal cone, we have
	\[
		\dual{\alpha^{-1}_k\lambda_k}{\alpha_k\chi_G(\bar u_k-\psi^u(\bar x_k))}{L^2(\Omega)}\geq 0,
		\quad
		\dual{\phi^\lambda(\bar x_k)}{\alpha_k\chi_G(\bar u_k-\psi^u(\bar x_k))}{L^2(\Omega)}\leq 0.
	\]
	Taking the limit $k\to\infty$, we thus obtain $\dual{\bar\lambda}{\chi_G\bar w}{L^2(\Omega)}=0$.
	Since $G\subset\Omega$ was chosen arbitrarily, \eqref{eq:CSt_w} follows.

	Finally, we are going to prove \eqref{eq:CSt_clarke}.
	Therefore, we fix an arbitrary measurable set $G\subset\Omega$.
	We first observe that
	due to \eqref{eq:multipliers_w} and \eqref{eq:multipliers_rho}, we have
	\[
		\dual{\opB^\star(p_k-\alpha_k\phi^p(\bar x_k))}{\alpha_k\chi_G(\bar u_k-\psi^u(\bar x_k))}{L^2(\Omega)}
		\to
		\dual{\opB^\star\bar\rho}{\chi_G\bar w}{L^2(\Omega)}.
	\]
	Now, we can exploit \eqref{eq:CSt_u}, \eqref{eq:difference_u},
	the weak sequential lower semicontinuity of the map 
	$L^2(\Omega)\ni u\mapsto \dual{\sigma u}{\chi_Gu}{L^2(\Omega)}\in\R$,
	and the definition of the normal cone in order to obtain
	\begin{align*}
		&\dual{-\bar\xi}{\chi_G\bar w}{L^2(\Omega)}
		=
		\dual{F'_u(\bar x,\bar y,\bar u)-\opB^\star\bar\rho}
			{\chi_G\bar w}{L^2(\Omega)}
			+\dual{\sigma\bar w}{\chi_G\bar w}{L^2(\Omega)}\\
		&\qquad\leq
			\lim\limits_{k\to\infty}
				\dual{F'_u(\bar x_k,\bar y_k,\bar u_k)-\opB^\star(p_k-\alpha_k\phi^\lambda(\bar x_k))}
				{\alpha_k\chi_G(\bar u_k-\psi^u(\bar x_k))}{L^2(\Omega)}\\
		&\qquad\qquad
			+\liminf\limits_{k\to\infty}
				\dual{\sigma\alpha_k(\bar u_k-\psi^u(\bar x_k))}{\alpha_k\chi_G(\bar u_k-\psi^u(\bar x_k))}{L^2(\Omega)}\\
		&\qquad=\liminf\limits_{k\to\infty}
			\dual{\alpha_k\phi^\lambda(\bar x_k)-\lambda_k}{\alpha_k\chi_G(\bar u_k-\psi^u(\bar x_k))}{L^2(\Omega)}
		\leq 0.
	\end{align*}
	Noting that $G\subset\Omega$ has been chosen arbitrarily, \eqref{eq:CSt_clarke} follows.
\end{proof}

Above, we have shown that the particular global minimizer $(\bar x,\bar y,\bar u)$ which results
from the relaxation approach suggested in \Cref{sec:upper_level} is C-stationary.
In order to carry over this analysis to arbitrary local minimizers of \eqref{eq:upper_level},
we exploit a localization argument.
\begin{theorem}\label{thm:CSt}
	Let $(\bar x,\bar y,\bar u)\in\R^n\times\YY\times L^2(\Omega)$ be a local minimizer of \eqref{eq:upper_level}.
	Then, it is C-stationary.
\end{theorem}
\begin{proof}
	Invoking \Cref{lem:existence_and_uniqueness_lower_level}, there is some $\varepsilon>0$ such that
	$\bar x$ is the unique globally optimal solution of 
	\[
		\begin{split}
			F(x,\psi^y(x),\psi^u(x))+\tfrac12\abs{x-\bar x}{2}^2&\,\to\,\min\limits_x\\
			x&\,\in\,X_\textup{ad}\cap\mathbb B^\varepsilon(\bar x)
		\end{split}
	\]
	where $\mathbb B^\varepsilon(\bar x)$ denotes the closed $\varepsilon$-ball around $\bar x$.
	Thus, $(\bar x,\bar y,\bar u)$ is the unique global minimizer of the bilevel programming problem
	\begin{equation}\label{eq:upper_level_localized}
		\begin{split}
		F(x,y,u)+\tfrac12\abs{x-\bar x}{2}^2&\,\to\,\min\limits_{x,y,u}\\
		x&\,\in\,X_\textup{ad}\cap\mathbb B^\varepsilon(\bar x)\\
		(y,u)&\,\in\,\Psi(x).
		\end{split}
	\end{equation}
	Combining \Cref{thm:relaxation_behavior} as well as \Cref{lem:convergence_of_multipliers,lem:CSt},
	$(\bar x,\bar y,\bar u)$ is a C-stationary point of \eqref{eq:upper_level_localized}.
	Noting that the derivative of the functional $\R^n\ni x\mapsto\tfrac12\abs{x-\bar x}{2}^2\in\R$
	vanishes at $\bar x$ while 
	$\mathcal N_{X_\textup{ad}\cap\mathbb B^\varepsilon(\bar x)}(\bar x)=\mathcal N_{X_\textup{ad}}(\bar x)$ holds since
	$\bar x$ is an interior point of $\mathbb B^\varepsilon(\bar x)$, the C-stationarity conditions
	of \eqref{eq:upper_level_localized} and \eqref{eq:upper_level} coincide at $(\bar x,\bar y,\bar u)$.
	This completes the proof.
\end{proof}
 
\begin{remark}\label{rem:M_stationarity}
		The counterexample from \cite[Section~3.2]{HarderWachsmuth2019} shows that the local minimizers
		of \eqref{eq:upper_level} are not S-stationary in general.
		However, it remains an open question whether the multipliers which solve the C-stationarity 
		system associated with a local minimizer of \eqref{eq:upper_level} additionally satisfy
		\[
		\begin{aligned}
			&\bar\xi\bar w\,=\,0\,\lor\,(\bar\xi\,>\,0\,\land\,\bar w\,>\,0)&
				&\text{a.e.\ on }\{\omega\in\Omega\,|\,\bar\lambda(\omega)=0\,
									\land\,\bar u(\omega)=u_a(\omega)\},&\\
			&\bar\xi\bar w\,=\,0\,\lor\,(\bar\xi\,<\,0\,\land\,\bar w\,<\,0)&
				&\text{a.e.\ on }\{\omega\in\Omega\,|\,\bar\lambda(\omega)=0\,
									\land\,\bar u(\omega)=u_b(\omega)\}.&
		\end{aligned}
		\]
		In line with the terminology of finite-dimensional complementarity programming,
		the resulting stationarity condition may be referred to as the system of (pointwise)
		Mordukhovich-stationarity.
		We would like to briefly note that this system \emph{cannot} be obtained by computing the
		limiting normal cone to the set $\gph\mathcal N_{U_\textup{ad}}$ since the latter turns
		out to be uncomfortably large, see \cite{MehlitzWachsmuth2018}. More precisely, this
		strategy results in the W-stationarity system of \eqref{eq:upper_level} from \cref{def:C_stationarity}. 
		Additionally, one cannot rely on the
		limiting variational calculus in $L^2(\Omega)$ due to an inherent lack of so-called
		sequential normal compactness, see \cite{Mehlitz2018}.
		Taking into account the outstanding success of variational analysis in the
		finite-dimensional setting, these observations are quite unexpected.
\end{remark}
%
%
%
\bibliographystyle{spmpsci}      
\bibliography{references}   

\end{document}